\documentclass[12pt]{amsart}
\usepackage[utf8]{inputenc}
\usepackage{parskip}
\usepackage[a4paper,
    left=1.5cm,
    right=1.5cm,
    top=1.8cm,
    bottom=1.8cm,
    footskip=.25in]{geometry}
\usepackage{amsmath, amssymb, amsthm}
\usepackage{mathtools}
\usepackage{hyperref}
\usepackage{comment}
\usepackage{import}
\usepackage{enumitem}
\usepackage{booktabs}
\usepackage{url}
\usepackage[table,xcdraw, dvipsnames]{xcolor}
\usepackage{pdfpages}
\usepackage{graphicx}
\usepackage{xcolor}
\usepackage{pinlabel}
\usepackage{tikz}
\usetikzlibrary{decorations.pathreplacing, positioning, calc}
\usepackage{caption}
\usepackage{subcaption}
\graphicspath{ {./image/} }
\usepackage{tikz-cd}
\tikzcdset{scale cd/.style={every label/.append style={scale=#1},
    cells={nodes={scale=#1}}}}
\newtheorem{thm}{Theorem}[section]
\newtheorem{theorem}[thm]{Theorem}

\newtheorem{lemma}[thm]{Lemma}
\newtheorem{prop}[thm]{Proposition}
\newtheorem{cor}[thm]{Corollary}

\theoremstyle{definition}

\newtheorem*{theorem*}{Theorem}
\newtheorem*{problem*}{Problem}

\def\R{\mathbb{R}}
\def\C{\mathbb{C}}
\def\F{\mathbb{F}}
\def\Z{\mathbb{Z}}

\def\res{\raise-.5ex\hbox{\ensuremath|}}

\makeatletter
\newcommand{\customlabel}[2]{%
   \protected@write \@auxout {}{\string \newlabel {#1}{{#2}{\thepage}{#2}{#1}{}} }%
   \hypertarget{#1}{#2}
}
\makeatother

\makeatletter
\def\@cite#1#2{{\normalfont[{#1\if@tempswa , #2\fi}]}}
\makeatother

\DeclareMathOperator{\Aut}{\mathrm{Aut}}
\DeclareMathOperator{\Id}{\mathrm{id}}
\DeclareMathOperator{\Img}{\mathrm{im}}

\DeclareMathOperator{\Spin}{\mathrm{Spin}}

\DeclareMathOperator{\Hom}{\mathrm{Hom}}
\DeclareMathOperator{\Map}{\mathrm{Map}}

\DeclareMathOperator{\Sp}{\mathrm{Sp}}
\DeclareMathOperator{\Mod}{\mathrm{Mod}} %Macro for the mapping class group
\DeclareMathOperator{\Diff}{\mathrm{Diff}}

\usepackage[style=alphabetic, backref=true]{biblatex}
\title[]{Representation stability for the first homology of congruence subgroups}
\author[]{Tudur Lewis}
\address{School of Mathematics, University of Bristol}
\email{et24788@bristol.ac.uk}

\begin{document}
\maketitle
\begin{abstract}
    We study sequences of modular representations of the symplectic and special linear groups over finite fields arising from the first homology of congruence subgroups of mapping class groups and automorphism groups of free groups, as well as from the module of coinvariants for the abelianization of the Torelli group. In each case, we determine the composition factors and their multiplicities, and establish periodic representation stability in the sense of Church--Farb. We apply our results to study flat line bundles over the moduli space of curves with level $2$ structure arising from spin structures on the underlying surface.
\end{abstract}
\tableofcontents
\section{Introduction}
Let $\Sigma_{g,1}$ be an oriented surface of genus $g$ with one boundary component. Let $\Mod_{g,1} = \pi_0(\mathrm{Diff}^+(\Sigma_{g,1}, \partial \Sigma_{g,1}))$ denote the mapping class group of $\Sigma_{g,1}$. Let $R$ be a ring, and let $\mathrm{Sp}_{2g}(R)$ denote the symplectic group over $R$, which we identify with the isometry group of the intersection form $(H_1(\Sigma_g;R), Q)$.

Let $F$ be a field of characteristic $2$. Let $\Mod_{g,1}[2]$ denote the kernel of the action of $\Mod_{g,1}$ on $H_1(\Sigma_{g,1};\Z/2)$. The conjugation action of $\Mod_{g,1}$ on $\Mod_{g,1}[2]$ induces an representation of $\Mod_{g,1}/\Mod_{g,1}[2] = \mathrm{Sp}_{2g}(\F_2)$ acting on $H_1(\Mod_{g,1}[2];F)$; see \cite[Ch.2.6]{brown_grpcoho} for more details.

We compute the composition factors of the $\mathrm{Sp}_{2g}(\F_2)$--module $H_1(\Mod_{g,1}[2];F)$. To state our result, note that the action of $\Mod_{g,1}$ on $H_1(\Sigma_g;F)$ factors to an action of $\Mod_{g,1}/\Mod_{g,1}[2] = \mathrm{Sp}_{2g}(\F_2)$ on $H_1(\Sigma_g;F)$. The following maps 
\begin{align*}
    \delta_k: \bigwedge^kH_1(\Sigma_g;F) \rightarrow \bigwedge^{k-2}H_1(\Sigma_g;F) \\
    w_1 \wedge \cdots \wedge w_k \mapsto \sum_{i<j}(-1)^{i+j-1}Q(w_i,w_j)w_1 \wedge \cdot \wedge \hat{w_i} \wedge \cdot \wedge \hat{w_j} \wedge \cdot \cdot \wedge w_k
\end{align*}
are $\mathrm{Sp}_{2g}(\F_2)$--equivariant, hence the $\ker (\delta_k)$ and the $\Img{(\delta_k)}$ give invariant subspaces. 

Let $\{ X_i \}_{i=1}^{2g}$ be a Symplectic basis for $H_1(\Sigma_g;F)$, that is, a basis for $H_1(\Sigma_g;F)$ that satisfies $$Q(X_{2i-1},X_{2j}) = \delta_{ij}, Q(X_{2i}, X_{2j}) = Q(X_{2i-1}, X_{2j-1}) = 0,$$ then $$\omega = \sum_{i=1}^g X_{2i-1} \wedge X_{2i}$$ is an invariant vector in $\bigwedge^2H_1(\Sigma_g;F)$. We obtain $\mathrm{Sp}_{2g}(\F_2)$--equivariant maps 
\begin{align*}
    \epsilon: \bigwedge^k H_1(\Sigma_g;F) \rightarrow \bigwedge^{k+2}H_1(\Sigma_g;F) \\
    x \mapsto \omega \wedge x
\end{align*}
so the $\Img{\epsilon}$ also give invariant subspaces.

\begin{theorem}
Let $F$ be a field of characteristic $2$, considered as a trivial $\Mod_{g,1}[2]$--module. The $\mathrm{Sp}_{2g}(\F_2)$--module $H_1(\Mod_{g,1}[2];F)$ has the following table of composition factors. Hence, for fixed parity of $g$, the composition factors stabilize for large $g$.

\begin{table}[h]
\begin{tabular}{|l|l|l|}
\hline
Genus                        & Composition factors                                                                                                                                                                                                                         & Multiplicities \\ \hline

$g \geq 4$ odd  & $H_1(\Sigma_g;F), F, \ker(\delta_2), \ker(\delta_3) / \Img{\epsilon}$                      & $3,1,1,1$        \\ \hline
$g \geq 4$ even & $H_1(\Sigma_g;F), F, \ker(\delta_2) / \langle \omega \rangle, \ker(\delta_3)$ & $2,2,1,1$        \\ \hline
\end{tabular}
\end{table}
\end{theorem}

For the proof, see Theorem \ref{comp_factors_ab_level2mod_field_char2}. The computation builds on the work of Sato \cite{sato}, and relies on a new description of $H_1(\Mod_{g,1}[2];\Z)$ as an algebra which makes the $\mathrm{Sp}_{2g}(\F_2)$ action more transparent. For genus three, the composition factors are obtained in Section \ref{exceptional_composition_factors_subsection}, for the case where $F = \F_2$ or its algebraic closure. The composition factors for the closed and once punctured cases are also worked out in Section \ref{oncepunctured_or_closed_level2_case}, by using the Birman exact sequence; see Lemmas \ref{oncepunctured_level2_sprep} and \ref{closedsurface_level2_sprep}.

Let $\mathcal{I}_{g,1}$ denote the Torelli subgroup of the mapping class group $\Mod_{g,1}$, that is, the kernel of the action of $\Mod_{g,1}$ on $H_1(\Sigma_g;\Z)$. We obtain an action of $\Mod_{g,1}$ on $H_1(\mathcal{I}_{g,1};F)$ induced by the conjugation action of $\Mod_{g,1}$ on $\mathcal{I}_{g,1}$.

Let $H_1(\mathcal{I}_{g,1};F)_{\Mod_{g,1}[2]}$ denote the module of coinvariants, that is, the largest quotient of $H_1(\mathcal{I}_{g,1};F)$ on which $\Mod_{g,1}[2]$ acts trivially. This is a $\Mod_{g,1}/\Mod_{g,1}[2] = \mathrm{Sp}_{2g}(\F_2)$--module. In Theorem \ref{jordan_holder_torelli_coinvariants}, we show the following.

\begin{theorem}
The composition factors of the $\mathrm{Sp}_{2g}(\F_2)$--module $H_1(\mathcal{I}_{g,1};F)_{\Mod_{g,1}[2]}$ are given by the following table, showing that for fixed parity of $g$, this representation also stabilizes.
\begin{table}[h]
\begin{tabular}{|l|l|l|}
\hline
genus & composition factors & multiplicities \\ \hline
   $g \geq 3$ even   &     $F,\; H_1(\Sigma_g;F), \; \ker{\delta_2}/\langle \omega \rangle,\; \ker{\delta_3}$      &  $3,2,1,1$              \\ \hline
    $g \geq 3$ odd  &       $F,\; H_1(\Sigma_g;F),\; \ker{\delta_2},\; \ker{\delta_3}/\Img{\epsilon}$           &        $2,2,1,1$        \\ \hline
\end{tabular}
\end{table}
\end{theorem}

Specializing $F =\F_2$ or its algebraic closure, we show in Theorem \ref{level2_stably_representation_periodic} that both families of representations given above are uniformly stably periodic in the sense of Church--Farb \cite[Def.8.1]{church_farb}. Our methods also show uniform stable periodicity for the sequence of $\mathrm{Sp}_{2g}(\F_2)$--modules given by $\{ H_1(\mathrm{Sp}_{2g}(\Z)[2];F)\}$, where $\mathrm{Sp}_{2g}(\Z)[2]$ is the level $2$ congruence subgroup of the symplectic group $\mathrm{Sp}_{2g}(\Z)$, that is, the kernel of the reduction map $\mathrm{Sp}_{2g}(\Z) \rightarrow \mathrm{Sp}_{2g}(\F_2)$. In Section \ref{level2_spgrp_section}, we show the following.
\begin{theorem} \label{five_term_level2_cong_mcg}
    Let $F$ be a field of characteristic $2$. Let $\rho:\Mod_{g,1} \rightarrow \mathrm{Sp}_{2g}(\Z)$ denote the symplectic representation obtained from the action of $\Mod_{g,1}$ on $H_1(\Sigma_g;\Z)$. For the five term exact sequence of $\mathrm{Sp}_{2g}(\F_2)$--modules $$H_1(\mathcal{I}_{g,1};F)_{\Mod_{g,1}[2]} \xrightarrow{\iota_*} H_1(\Mod_{g,1}[2];F) \xrightarrow{\rho_*} H_1(\mathrm{Sp}_{2g}(\Z)[2];F) \rightarrow 0$$ associated to the exact sequence $$1 \rightarrow \mathcal{I}_{g,1} \xrightarrow{\iota} \Mod_{g,1}[2] \xrightarrow{\rho} \mathrm{Sp}_{2g}(\Z)[2] \rightarrow 1$$
    we have that $\Img{(\iota_*)} \cong \wedge^3 H_1(\Sigma_g;F)$ as an $\mathrm{Sp}_{2g}(\F_2)$--module.
\end{theorem}

This allows us to find the composition factors for the Jordan--Holder series of these representations; see Corollary \ref{compfactors_ab_level2_sp}. Note that these results for the sequence $\{ H_1(\mathrm{Sp}_{2g}(\Z)[2];\F_2)\}$ was first obtained by Putman; see \cite[Section 8.3]{church_farb} and \cite[Section 3]{putmanduke}. 

\subsection*{Applications to flat line bundles on moduli space}
Let $\Sigma_{g,b}^p$ denote a compact oriented surface of genus $g$ with $b$ boundary components and $p$ marked points in its interior. Let $\mathrm{Diff}^+(\Sigma_{g,b}^p, \partial \Sigma_{g,b}^p)$ denote the topological group consisting of orientation-preserving diffeomorphisms of $\Sigma_{g,b}^p$ that fix the boundary pointwise and permute the marked points. Let $\Mod_{g,b}^p = \pi_0(\mathrm{Diff}^+(\Sigma_{g,b}^p, \partial \Sigma_{g,b}^p))$ denote the mapping class group. Let $\mathcal{T}_{g,b}^p$ denote the Teichmuller space of $\Sigma_{g,b}^p$; this gives a contractible space on which $\Mod_{g,b}^p$ acts properly discontinuously. Define $\Mod_{g,b}^p[2]$ to be the kernel of the action of $\Mod_{g,b}^p$ on $H_1(\Sigma_{g,b};\F_2)$.

Suppose $b+p \leq 1$. In Section \ref{line_bundles_section} we consider the \textit{moduli space of curves with level $2$-structure}, that is, the orbifold $(\mathcal{T}_{g,b}^p, \Mod_{g,b}^p[2])$, along with the induced action of a subgroup $\mathrm{Stab}_{g,b}^p[q]$ of $\Mod_{g,b}^p$ that fixes a spin structure; spin structures on $\Sigma_g$ are in natural bijection with the set of \textit{Symplectic quadratic forms} $q:H_1(\Sigma_g;\F_2) \rightarrow \F_2$ that satisfy $q(x+y) = q(x) + q(y) + Q(x, y)$ for all $x,y \in H_1(\Sigma_g;\F_2)$.

In Lemma \ref{spin_mcg_fixes_flat_line_bundle}, we use our description of $H_1(\Mod_{g,1}[2];\F_2)$ to construct a flat line bundle out of a Symplectic quadratic form $q$ such that the action of $\mathrm{Stab}_{g,1}[q]$ by pullback fixes this line bundle. Moreover, in Theorem \ref{spin_mcg_flat_line_uniqueness} we show that the flat line bundle constructed in Lemma \ref{spin_mcg_fixes_flat_line_bundle} is the unique nonzero torsion element of $\mathrm{Pic}(\mathcal{T}_{g,1}, \Mod_{g,1}[2])$ fixed by $\mathrm{Stab}_{g,1}[q]$. Hence we obtain the following.
\begin{theorem}
    Suppose that $(b,p) = (1,0)$ and $g \geq 3$, or suppose that $b+p \leq 1$ and $g \geq 9$. Let $q:H_1(\Sigma_g;\F_2) \rightarrow \F_2$ be a Symplectic quadratic form. There exists a unique nontrivial flat line bundle in the orbifold Picard group $\mathrm{Pic}(\mathcal{T}_{g,b}^p, \Mod_{g,b}^p[2])$ left invariant under the action of $\mathrm{Stab}_{g,b}^p[q]$ by pullback.
\end{theorem}

\subsection*{Representation stability for congruence subgroups of $\mathrm{Aut}(F_n)$}

Next, we consider representations obtained from the homology of the level $p$ congruence subgroup of the automorphism group of a free group. Let $F_n$ be a free group of rank $n$, and let $\Aut (F_n)$ denote the automorphism group of $F_n$. Let $\Aut(F_n)[p]$ denote the kernel of the action of $\Aut(F_n)$ on $H_1(F_n;\Z/p)$. The conjugation action of $\Aut(F_n)$ on $\Aut(F_n)[p]$ gives an action of $\Aut(F_n)/\Aut(F_n)[p] = \mathrm{GL}_n(\F_p)$ on $H_1(\Aut(F_n)[p];\Z/p)$. We obtain a representation stability result for the $\mathrm{SL}_n(\F_p)$--modules $H_1(\Aut(F_n)[p];\Z/p)$.

To state these results, we label simple modules using the highest weight theory of algebraic groups. Let $p$ be a prime, let $\F_p$ be the field with $p$ elements, and let $\overline{\F}_p$ denote its algebraic closure. Fix a maximal torus $T$ of $\mathrm{SL}_n\overline{\F}_p$ to be the subgroup of diagonal matrices, and let $X(T)$ denote the character group of $T$. Consider the identification $X(T) = \Z[L_1,..,L_n]/ \langle L_1+ \cdots + L_n \rangle$ where $L_i:T \rightarrow \overline{\F}_p^*$ sends $\mathrm{diag}(c_1,..,c_n)$ to $c_i$, and write multiplication of characters using additive notation. Let $\omega_i = L_1 + \cdots + L_i$ denote fundamental weights. Call a weight $\lambda \in X(T)$ dominant if it can be written as a nonnegative integral linear combination of the fundamental weights. For a dominant weight $\lambda$, denote the rational irreducible $\mathrm{SL}_n$--module with highest weight $\lambda$ by $L(\lambda)$; see \cite[Thm. 2.2 and 2.11]{humphreys_modular_reps}. In Corollary \ref{congruence_autfn_composition_factors}, we obtain the following.

\begin{thm}
    Let $p$ be a prime. For the $\mathrm{SL}_n(\F_p)$--modules $H_1(\Aut(F_n)[p];\F_p)$ and $H_1(\Aut(F_n)[p];\overline{\F}_p)$, we have the following.
    \begin{itemize}
        \item If $n = 1 \pmod{p}$, then the composition factors are $L(\omega_1), L(\omega_2 + \omega_{n-1}), L(\omega_1), L(\omega_1 + \omega_{n-1})$.
        \item If $n = 0 \pmod{p}$, then the composition factors are $L(\omega_2+\omega_{n-1}), L(\omega_1), L(0), L(\omega_1 + \omega_{n-1})$.
        \item If $n \neq 0,1 \pmod{p}$, then the composition factors are $L(\omega_2+\omega_{n-1}), L(\omega_1) , L(\omega_1 + \omega_{n-1})$.
    \end{itemize}
    Hence these modules give a stably periodic sequence in the sense of Church--Farb \cite[Def.8.1]{church_farb}.
\end{thm}
For a concrete description of the modules $L(\omega_1 + \omega_{n-1})$ and $L(\omega_2 + \omega_{n-1})$ see Lemma \ref{compo_factor_SL_traceless_matrices} and Corollary \ref{coinvariants_autfn_composition_factors} respectively.

\subsection{Related work}
The homology of congruence subgroups of mapping class groups and automorphism groups of free groups is an area of active interest. Calculations by Sato, Putman, and Satoh, imply that the sequence of groups $H_1(\Mod_{g,1}[p];\F_p)$ and $H_1(\Aut(F_n)[p];\F_p)$ are not homologically stable; see \cite[Thm.H]{putmanduke} and \cite{satoh_congruence}.

Sato's calculation \cite{sato} of $H_1(\Mod_{g,1}[2];\Z)$ required separate methods from the $p>2$ prime case, and relates to the signature of $4$--manifolds via the Rochlin invariant. Previous work of the author \cite{first_paper} related this calculation to the Birman--Craggs maps of the Torelli group, and to Meyer's signature cocycle.

Work of Putman--Sam \cite[Theorems I and K]{putman_sam} shows that the $\mathrm{SL}_n$--modules $H_1(\Aut(F_n)[p];\F_p)$, and the $\mathrm{Sp}_{2g}$--modules $H_1(\Mod_{g,1}[p];\F_p)$ satisfy a form of representation stability. Our contribution is to show that these sequences of representations give a stably periodic sequence of modular representations in the sense of Church--Farb \cite[Sect.8]{church_farb}. We do this by calculating the composition factors for both sequences, allowing a calculation of the periods of the multiplicities of the modules $L(\lambda)$ as composition factors.

\subsection{Outline}
In Section \ref{mod2_abel_section} we give a new description of $H_1(\Mod_{g,1}[2];\Z)$ as an algebra that is used later to study the $\mathrm{Sp}_{2g}(\F_2)$--module structure. In Section \ref{Sp_rep_section} we calculate the composition factors for the $\mathrm{Sp}_{2g}(\F_2)$--module $H_1(\Mod_{g,1}[2];F)$. In Section \ref{bcj_section} we calculate the composition factors for the $\mathrm{Sp}_{2g}(\F_2)$--module $H_1(\mathcal{I}_{g,1};F)_{\Mod_{g,1}[2]}$. In Section \ref{rep_stab_sp_section} we show that the calculations of Sections \ref{Sp_rep_section} and \ref{bcj_section} imply that the sequences of $\mathrm{Sp}_{2g}(\F_2)$--modules studied are uniformly stably periodic.

In Section \ref{bp_sep_twist_eval_subsection} we compute the image of the Torelli subgroup in $H_1(\Mod_{g,1}[2];\F_2)$; see Corollary \ref{img_torelli_mod2}. In Section \ref{oncepunctured_or_closed_level2_case}, we use the results in Section \ref{bp_sep_twist_eval_subsection}, along with the Birman exact sequence, to obtain the composition factors for the $\mathrm{Sp}_{2g}(\F_2)$--modules obtained from the level $2$ congruence subgroup in the case of once punctured and closed surfaces. Section \ref{bp_sep_twist_eval_subsection} is also used in Section \ref{level2_spgrp_section} to show Theorem \ref{five_term_level2_cong_mcg}. This computes the composition factors of the $\mathrm{Sp}_{2g}(\F_2)$--module $H_1(\mathrm{Sp}_{2g}(\Z)[2];F)$; see Corollary \ref{compfactors_ab_level2_sp}. In Section \ref{line_bundles_section} we use the previous sections to study flat line bundles in $\mathrm{Pic}(\mathcal{T}_{g,b}^p, \Mod_{g,b}^p[2])$.

In Section \ref{cong_autfn_section} we calculate the composition factors for the $\mathrm{SL}_n(\F_p)$--module $H_1(\mathrm{Aut}(F_n)[p];\F_p)$, showing that this sequence is stably periodic; see Corollary \ref{congruence_autfn_composition_factors}.

\subsection*{Acknowledgments} The author thanks Tara Brendle and Andrew Putman for helpful discussions. The author is grateful to Benson Farb for suggesting the central problem of this paper, which grew out of a question he posed about the level $2$ congruence subgroup of the mapping class group.

\section{The abelianization of the level 2 congruence group} \label{mod2_abel_section}

In this section, we give an alternative description of $H_1(\Mod_{g,1}[2] ; \mathbb{Z})$ by defining a family of polynomial algebras indexed by a spin structure $\sigma \in \Spin(\Sigma_g)$.
We rewrite Sato's homomorphisms as maps from $\Mod_{g,1}[2]$ to this algebra. Then we use a single relation in this polynomial algebra to calculate the image of Sato's maps. We begin by recalling Sato's results \autocite{sato}. Throughout the paper, $H_1(\Sigma_g)$ is always assumed with $\Z/2$ coefficients, unless specified otherwise. For the intersection form $Q$ on $H_1(\Sigma_g)$, we abbreviate $Q(v,w) = v \cdot w$.

Let $\mathrm{Spin}(\Sigma_g)$ denote the set of spin structures on $\Sigma_g$. Johnson shows there is a natural bijection between spin structures and \textit{symplectic quadratic forms} in \cite[Thm.3A,3B]{johnsonspin}. Here, a symplectic quadratic form is a map $q:H_1(\Sigma_g) \rightarrow \Z/2$ that satisfies $q(x+y) = q(x) + q(y) + x \cdot y$ for all $x,y \in H_1(\Sigma_g)$.

For $z \in H_1(\Sigma_g)$, define $i_z: H_1(\Sigma_g) \rightarrow \mathbb{Z}/8$ by

\begin{equation*}
i_z(y) = 
   \left\{
\begin{array}{ll}
      1 , & z \cdot y =1 \mod 2 \\
      0 , & z \cdot y =0 \mod 2 .
\end{array} 
\right. 
\end{equation*}

This function is not a homomorphism, but we have the following two identities

\begin{equation}
    \label{identity1}
    i_a(x)+i_b(x)-2i_a(x)i_b(x) = i_{a+b}(x),
\end{equation}

\begin{equation}
    \label{identity2}
    ((-1)^{q_{\sigma}(C)}i_C(x))^2=(-1)^{q_{\sigma}(C)}((-1)^{q_{\sigma}(C)}i_C(x))
\end{equation}

where $q_{\sigma}$ is the symplectic quadratic form associated to $\sigma \in \Spin(\Sigma_g)$, following the notation in \cite[Thm.2.1]{first_paper}. Both identities are checked by comparing both sides of the equation elementwise. We use the following results of Sato.

\begin{prop} \autocite[Lem. 2.2, Prop. 5.2, Cor.8.9]{sato}
\label{Satomainresult}
Let $g \geq 3$. Let $\sigma \in Spin(\Sigma_g)$ be a spin structure, and let $q_{\sigma}$ be the associated quadratic form on $H_1(\Sigma_g)$. Let $C \subset \Sigma_g - D$ be a non-separating simple closed curve on $\Sigma_g$ whose image does not intersect a fixed embedded disk $D \subset \Sigma_g$. The map 
\begin{align*}
    \beta_{\sigma}: \Mod_{g,1}[2] \rightarrow \Map(H_1(\Sigma_g), \mathbb{Z}/8) \\
    t_C^2 \mapsto (-1)^{q_{\sigma}(C)}i_{[C]},
\end{align*}
is a homomorphism that induces an injective homomorphism $\beta_{\sigma}: H_1(\Mod_{g,1}[2];\Z) \rightarrow \Map(H_1(\Sigma_g),\Z/8)$. Here, $t_C \in \Mod_{g,1}$ denotes the Dehn twist about $C$.
\end{prop}

Recall that $\Map(H_1(\Sigma_g), \mathbb{Z}/8)$ is the free $\mathbb{Z}/8$-module consisting of all functions $H_1(\Sigma_g) \rightarrow \Z/8$. If we add the operation of function multiplication, this turns $\Map(H_1(\Sigma_g), \mathbb{Z}/8)$ into a $\mathbb{Z}/8$-algebra. Lemma \ref{algebra_relations_lemma} gives the two relations we use to analyse the image of Sato's maps; note that $\Mod_{g,1}[2]$ is generated by squares of Dehn twists by Humphries \cite[Prop.2.1]{humphriesfactorisation}, hence to find the image of $\beta_{\sigma}$, we must compute the span of the $\beta_{\sigma}(t_C^2)$ in $\Map(H_1(\Sigma_g),\Z/8)$.

\begin{lemma} \label{algebra_relations_lemma}
    Let $\sigma \in \Spin(\Sigma_g)$ and denote by $q_{\sigma}$ the associated quadratic form. Then in the subalgebra $W_{\sigma}$ of $\Map(H_1(\Sigma_g), \mathbb{Z}/8)$ generated by $\overline{C}= (-1)^{q_{\sigma}(C)}i_C$ for all $C \in H_1(\Sigma_g)$, the following relations hold:
    \begin{enumerate}
        \item $\overline{C_1 + C_2} = (-1)^{C_1 \cdot C_2}((-1)^{q_{\sigma}(C_2)}\overline{C_1} + (-1)^{q_{\sigma}(C_1)}\overline{C_2}-2\overline{C_1} \overline{C_2})$, for all $C_1 \neq C_2 \in H_1(\Sigma_g)$.
        \item $\overline{C}^2 = (-1)^{q_{\sigma}(C)} \overline{C}$, for all $C \in H_1(\Sigma_g)$.
    \end{enumerate}
\end{lemma}
\begin{proof}
Using the identity (\ref{identity1}), we have 
\begin{multline}
    \label{manip1}
    (-1)^{q_{\sigma}(C_1 + C_2)}i_{C_1 + C_2}(x) = (-1)^{q_{\sigma}(C_1 + C_2)}(i_{C_1}(x)+i_{C_2}(x) - 2i_{C_1}(x)i_{C_2}(x)) \\
    =(-1)^{q_{\sigma}(C_2) + C_1 \cdot C_2}((-1)^{q_{\sigma}(C_1)}i_{C_1}(x)) + (-1)^{q_{\sigma}(C_1) + C_1 \cdot C_2}((-1)^{q_{\sigma}(C_2)}i_{C_2}(x)) \\
    - 2(-1)^{q_{\sigma}(C_1+C_2)}i_{C_1}(x)i_{C_2}(x).
\end{multline}
We also have that 
\begin{equation}
    \label{manip2}
    (-1)^{q_{\sigma}(C_1)}i_{C_1}(x)(-1)^{q_{\sigma}(C_2)}i_{C_2}(x) = (-1)^{q_{\sigma}(C_1 + C_2) - C_1 \cdot C_2}i_{C_1}(x)i_{C_2}(x).
\end{equation}
Set $\overline{C} = (-1)^{q_{\sigma}(C)}i_C$ and substitute (\ref{manip2}) into the last line of (\ref{manip1}) to get,
\begin{align*}
    \overline{C_1 + C_2} = (-1)^{q_{\sigma}(C_2) + C_1 \cdot C_2}\overline{C_1} +  (-1)^{q_{\sigma}(C_1) + C_1 \cdot C_2}\overline{C_2} - 2 (-1)^{C_1 \cdot C_2} \overline{C_1} \overline{C_2}\\
    =(-1)^{C_1 \cdot C_2}((-1)^{q_{\sigma}(C_2)}\overline{C_1} + (-1)^{q_{\sigma}(C_1)}\overline{C_2}-2\overline{C_1} \overline{C_2}).
\end{align*}
\end{proof}

Fix a symplectic basis $ \{a_1, b_1,....,a_g,b_g\}$ for $H_1(\Sigma_g)$ with $a_i \cdot b_j = \delta_{ij}$, $a_i \cdot a_j = 0$ and $b_i \cdot b_j = 0$. Any $C \in H_1(\Sigma_g)$ can be written uniquely as a linear combination of the $a_i, b_i$'s, so iteratively applying the relations of Lemma \ref{algebra_relations_lemma}, we see that any element in the $\mathbb{Z}/8$-submodule $\Img{\beta_{\sigma}}$ spanned by the $\overline{C_i}$'s can be written as a linear combination of elements of the form 
\begin{equation}
    \label{imagebasis}
    \overline{X_i}, 2\overline{X_i}\overline{X_j}, 4\overline{X_i}\overline{X_j}\overline{X_k},
\end{equation}
where $X_i,X_j$ and $X_k$ are distinct elements from $B$. To show that 
\begin{equation*}
    \Img{\beta_{\sigma}} = \mathbb{Z}/8^{2g} \bigoplus 2\mathbb{Z}/8^{2g \choose 2} \bigoplus 4\mathbb{Z}/8^{2g \choose 3}
\end{equation*}
as a module, we prove the following.

\begin{prop}
\label{basisforabelianization}
Suppose $$\sum_{i<j<k}(a_i\overline{X_i} + b_{ij}2 \overline{X_i} \; \overline{X_j} + c_{ijk}4 \overline{X_i} \; \overline{X_j} \; \overline{X_k}) = 0,$$ where $a_i, b_{ij}, c_{ijk} \in \Z/8$, and the $X_i$ are elements of the symplectic basis $B$. Then we have $a_i = 0, 2b_{ij}=0, 4c_{ijk}=0$ for all $i<j<k$.
\end{prop}
\begin{proof}
Suppose that we had such a linear combination $f \in \Map(H_1(\Sigma_g), \mathbb{Z}/8)$ that was equal to the constant zero function. Since $B$ is a symplectic basis, the definition of the functions $i_{[C]}$ gives us the following.

(i) If $X_i \in B$ then $i_{X_i}(X_j) =1$ for the unique $X_j \in B$ with $X_i \cdot X_j = \pm 1$, and it is $0$ on every other element of $B$. This implies that for distinct $X_i, X_j,X_k$ in $B$ we have that 
$i_{X_i}(X_l)i_{X_j}(X_l) = i_{X_i}(X_l)i_{X_j}(X_l)i_{X_k}(X_l) = 0$ for any $X_l \in B$.

Using (i), if we evaluate $f$ on all elements of $B$ we see that all the monomial coefficients are zero. So $f$ has no monomial terms.

(ii) Similarly we have that if $X_i, X_j, X_k, X_l \in B$ with $X_i, X_j$ distinct, then $i_{X_i}(X_l + X_k)i_{X_j}(X_l+X_k)$ can only be nonzero if up to reordering indices we have that $X_i \cdot X_l = \pm 1$ and $X_j \cdot X_k = \pm 1$. Since $B$ is a symplectic basis this tuple $(X_l, X_k)$ is unique, so for $X_i, X_j, X_k \in B$ distinct, we have that $i_{X_i}(X_l + X_n)i_{X_j}(X_l + X_n)i_{X_k}(X_l + X_n)$ is zero for all $X_l, X_n \in B$ and $i_{X_i}(X_l + X_k)i_{X_j}(X_l+X_k)$ is only non-zero for one element $X_r + X_s$.

Using (ii), evaluate $f$ on all elements of the form $X_l + X_k$ to get that all the coefficients of the quadratic terms of $f$ must be either $0$ or zero divisors, and so $f$ can only contain cubic terms.

(iii) If $X_i, X_j, X_k \in B$ are distinct then $i_{X_i}(X_a + X_b + X_c)i_{X_j}(X_a + X_b + X_c)i_{X_k}(X_a + X_b + X_c)$ can only be nonzero if, up to reordering, we have that $X_i \cdot X_a = \pm1$, $X_j \cdot X_b = \pm 1$, $X_k \cdot X_c = \pm 1$. Since $B$ is a symplectic basis, this sum is unique, so evaluating $f$ on all elements of the form $X_a + X_b + X_c$ gives the result.
\end{proof}
Combining Propositions \ref{Satomainresult} and \ref{basisforabelianization} gives
\begin{equation*}
    H_1(\Mod_{g,1}[2];\mathbb{Z}) = \mathbb{Z}/8^{2g} \bigoplus 2\mathbb{Z}/8^{2g \choose 2} \bigoplus 4\mathbb{Z}/8^{2g \choose 3}.
\end{equation*}

\section{Action on homology of the level $2$ congruence subgroup} \label{Sp_rep_section}
Now we consider the induced representation of $\mathrm{Sp}_{2g}(\F_2)$ on $H_1(\Mod_{g,1}[2];\Z/2)$. Denote the image of $$\beta_{\sigma}: \Mod_{g,1}[2] \rightarrow \Map(H_1(\Sigma_g), \Z/8)$$ by $W_{\sigma,g}$. The conjugation action of $\Mod_{g,1}$ on $\Mod_{g,1}[2]$ turns $H_1(\Mod_{g,1}[2];\Z)$ into a $\Mod_{g,1}/\Mod_{g,1}[2] \cong \mathrm{Sp}_{2g}(\F_2)$--module. Define an action of $\mathrm{Sp}_{2g}(\F_2)$ on $W_{\sigma,g}$ via $f \cdot \overline{C} = \overline{f(C)}$, for $f \in \mathrm{Sp}_{2g}(\F_2)$ and $C \in H_1(\Sigma_g)$. Then $\beta_{\sigma}: H_1(\Mod_{g,1}[2];\Z) \rightarrow W_{\sigma,g}$ is $\mathrm{Sp}_{2g}(\F_2)$--equivariant.

\begin{lemma} \label{rep_is_faithful}
    The representation of $\mathrm{Sp}_{2g}(\F_2)$ acting on $W_{\sigma, g}$ and $W_{\sigma, g} \otimes_{\Z}\Z/2$ is faithful.
\end{lemma}
\begin{proof}
    Let \begin{align*}
        \Psi: H_1(\Sigma_g) \rightarrow H_1(\Sigma_g)^* \\
        a \mapsto Q(a, -)
    \end{align*}
    denote the isomorphism obtained from the currying map of the intersection form. Suppose $f \in \mathrm{Sp}_{2g}(\F_2)$ acts trivially on $W_{\sigma,g}$, so $f \cdot \overline{X_j} = \overline{X_j}$ for all $X_j \in S_g$. 
    
    Since $\overline{X_j} = (-1)^{q_{\sigma}(X_j)}i_{X_j}$, and the function $i_{X_j}$ is induced by the intersection form, we get $X_j \cdot z = f(X_j) \cdot z$, for all $z \in H_1(\Sigma_g)$. Injectivity of $\Psi$ then gives $f(X_j) = X_j$, for all $X_j \in S_g$, hence $f = \mathrm{id}$, showing the first claim. Since the $\overline{X_j}$ have distinct images in $W_{\sigma, g} \otimes_{\Z}\Z/2$, the last claim follows.
\end{proof}

\subsection{Jordan--Holder decomposition}
Now we compute the representation of $\mathrm{Sp}_{2g}(\F_2)$ acting on $H_1(\Mod_{g,1}[2];\Z/2)$ explicitly. We write the representation in terms of the basis $\{ \overline{X_i}, 2\overline{X_i}\overline{X_j}, 4\overline{X_i}\overline{X_j}\overline{X_k}\}$ considered above.

For $\overline{C} = (-1)^{q_{\sigma}(C)}i_C \in W_{\sigma,g}$ and $a \in \Z/2$ we have $$(\overline{aC})\otimes 1 = a (\overline{C} \otimes 1),$$ $$(2\overline{C}^2) \otimes 1 = 2(\overline{C}\otimes 1) = 0,$$ and $$(4\overline{C}^2 \overline{D})\otimes 1 = 2((2\overline{C}\overline{D}) \otimes 1) = 0$$ for all $C,D \in H_1(\Sigma_g)$. We use these relations to compute the action of $\mathrm{Sp}_{2g}(\F_2)$ on $W_{\sigma, g} \otimes_{\Z} \Z/2$. By using the relations of Lemma \ref{algebra_relations_lemma}, and noting that the signs do not matter after tensoring with $\Z/2$, we obtain the following computation.

\begin{lemma}\label{general_matrix_formula}
    Let $f \in \mathrm{Sp}_{2g}(\F_2)$, and let $\{X_i\}_{i=1}^{2g}$ be a basis for $H_1(\Sigma_g)$. Let $f(X_j) = \sum_i a_{ij}X_i$, then in $W_{\sigma,g} \otimes_{\Z} \Z/2$ we have
    $$\overline{f(X_j)} = \sum_i a_{ij}\overline{X_i} + \sum_{i <k}a_{ij}a_{kj}(2\overline{X_i}\overline{X_k})+ \sum_{i<k<l}a_{ij}a_{kj}a_{lj}(4\overline{X_i}\overline{X_k}\overline{X_l}).$$
\end{lemma}

The following lemma is used to obtain the composition factors for $W_{\sigma, g} \otimes_{\Z} \Z/2$.

\begin{lemma}\label{level2ab_filtration_submod}
    Let $g \geq 2$. Let $Z_g$ denote the span of all vectors of the form $(4\overline{u} \; \overline{v} \; \overline{w}) \otimes 1$ in $W_{\sigma,g} \otimes_{\Z} \Z/2$, where $u,v,w \in H_1(\Sigma_g)$ and $\overline{u} = (-1)^{q_{\sigma}(u)}i_u$. Let $Q_g$ denote the span of all vectors of the form $(2\overline{u} \; \overline{v}) \otimes 1$ in $W_{\sigma, g} \otimes_{\Z} \Z/2$, where $u,v \in H_1(\Sigma_g)$. Then we have a filtration by $\mathrm{Sp}_{2g}(\F_2)$--submodules $$0 \subset Z_g \subset Q_g \subset W_{\sigma,g} \otimes_{\Z} \Z/2$$
    with $Z_g \cong \wedge^3H_1(\Sigma_g), \; Q_g/Z_g \cong \wedge^2 H_1(\Sigma_g),$ and $W_{\sigma, g} \otimes_{\Z} \Z/2 / Q_g \cong H_1(\Sigma_g)$ as $\mathrm{Sp}_{2g}(\F_2)$--modules.
\end{lemma}
\begin{proof}
    Using the relations of Lemma \ref{algebra_relations_lemma}, and the linearity of the action of $\mathrm{Sp}_{2g}(\F_2)$ on $W_{\sigma,g}$, we have \begin{align*}
        f \cdot (2 \overline{u} \; \overline{v})\otimes1 = (2\overline{f(u)} \; \overline{f(v)}) \otimes 1 \\
        f \cdot (4 \overline{u} \; \overline{v} \; \overline{w}) \otimes 1 = (4\overline{f(u)} \; \overline{f(v)} \; \overline{f(w)}) \otimes 1 \\
        \overline{u+v} \otimes 1 = \overline{u} \otimes 1 + \overline{v} \otimes 1 + (2\overline{u} \; \overline{v})\otimes 1
    \end{align*}
    for all $f \in \Sp_{2g}(\F_2)$ and $u,v,w \in H_1(\Sigma_g)$. Hence $Z_g$ and $Q_g$ are submodules with $Z_g \subset Q_g$. The relations of Lemma \ref{algebra_relations_lemma} also imply that the map \begin{align*}
        H_1(\Sigma_g) \times H_1(\Sigma_g) \times H_1(\Sigma_g) \rightarrow Z_g \\
        (u,v,w) \mapsto (4\overline{u} \; \overline{v} \; \overline{w})\otimes 1
    \end{align*}
    is $\F_2$--multilinear and alternating, hence it induces an $\F_2$--linear map $\Psi: \wedge^3 H_1(\Sigma_g) \rightarrow Z_g$. Now, $Z_g$ has basis $(4\overline{X_i} \; \overline{X_j} \; \overline{X_k}) \otimes 1$ for $i<j<k$ with $X_i,X_j,X_k \in B$, so $Z_g$ has dimension $\binom{2g}{3}$. Since $\Psi$ is surjective, we see that $\Psi$ is an isomorphism of $\Sp_{2g}(\F_2)$--modules.

    Since \begin{align} \label{quad_mod_cubic_relation}
        (2\overline{u+v} \; \overline{w}) \otimes 1 = (2 \overline{u} \; \overline{w}) \otimes 1 + (2 \overline{v} \; \overline{w}) \otimes 1 + (4 \overline{u} \; \overline{v} \; \overline{w})\otimes 1,
    \end{align}
    we see that $Q_g/Z_g$ has basis the $(2\overline{X_i} \; \overline{X_j}) \otimes 1 + Z_g$ for $i<j$ with $X_i,X_j \in B$, hence it has dimension $\binom{2g}{2}$. The relation (\ref{quad_mod_cubic_relation}) implies that \begin{align*}
        H_1(\Sigma_g) \times H_1(\Sigma_g) \rightarrow Q_g/Z_g \\
        (u,v) \mapsto (2\overline{u} \; \overline{v})\otimes 1 + Z_g
    \end{align*}
    is $\F_2$--bilinear and alternating, hence it induces $\xi: \wedge^2 H_1(\Sigma_g) \rightarrow Q_g/Z_g$. Since $\xi$ is surjective and $\Sp_{2g}(\F_2)$--equivariant, it is an isomorphism of $\Sp_{2g}(\F_2)$--modules.

    Similarly, the map \begin{align*}
        \varphi: H_1(\Sigma_g) \rightarrow W_{\sigma, g} \otimes_{\Z} \Z/2 / Q_g \\
        u \mapsto \overline{u} \otimes 1 + Q_g
    \end{align*}
    is $\F_2$--linear and $\Sp_{2g}(\F_2)$--equivariant, by the relations of Lemma \ref{algebra_relations_lemma}. Since $\varphi$ is surjective, and $W_{\sigma, g} \otimes_{\Z} \Z/2 / Q_g$ has dimension $2g$, $\varphi$ must be an isomorphism of $\Sp_{2g}(\F_2)$--modules.
\end{proof}

\subsection{Contraction maps}
By Lemma \ref{level2ab_filtration_submod}, the composition factors of the $\mathrm{Sp}_{2g}(\F_2)$--modules $\bigwedge^2 H_1(\Sigma_g)$ and $\bigwedge^3 H_1(\Sigma_g)$ give the composition factors of the $\mathrm{Sp}_{2g}(\F_2)$--module $W_{\sigma, g} \otimes_{\Z} \Z/2$. Let $F$ be a field of characteristic $2$, then we obtain a field extension $\F_2 = \{0,1\} \subset F$, then the filtration $$0 \subset Z_g \otimes_{\F_2} F \subset Q_g \otimes_{\F_2} F \subset (W_{\sigma, g} \otimes_{\Z} \Z/2)\otimes_{\F_2} F$$
obtained from Lemma \ref{level2ab_filtration_submod} gives a filtration of $H_1(\Mod_{g,1}[2];F)$ by $\mathrm{Sp}_{2g}(\F_2)$--submodules. Since $- \otimes_{\F_2} F$ is exact, we obtain $Z_g \otimes_{\F_2} F \cong (\wedge^3 H_1(\Sigma_g))\otimes_{\F_2} F = \wedge^3 (H_1(\Sigma_g) \otimes_{\F_2} F)$, $Q_g \otimes_{\F_2} F / Z_g \otimes_{\F_2} F \cong \wedge^2 (H_1(\Sigma_g) \otimes_{\F_2}F)$, and $(W_{\sigma, g} \otimes_{\Z} \Z/2)\otimes_{\F_2} F / Q_g \otimes_{\F_2} F \cong H_1(\Sigma_g) \otimes_{\F_2} F$ as $\mathrm{Sp}_{2g}(\F_2)$--modules. By the universal coefficient theorem, there is a natural isomorphism $H_1(\Sigma_g) \otimes_{\F_2} F \cong H_1(\Sigma_g;F)$, so for $f \in \Mod_{g,1}$, the induced action of $[f] \in \Mod_{g,1}/\Mod_{g,1}[2] = \mathrm{Sp}_{2g}(\F_2)$ on $H_1(\Sigma_g;F)$ agrees with that on $H_1(\Sigma_g;\F_2) \otimes_{\F_2} F$. This motivates studying the $\mathrm{Sp}_{2g}(\F_2)$--modules $\wedge^2 H_1(\Sigma_g;F)$ and $\wedge^3 H_1(\Sigma_g;F)$.

Let $F$ be a field of characteristic $2$, let $V_g$ be a vector space over $F$ of dimension $2g$, and let $-\cdot-: V_g \times V_g \rightarrow F$ be a non--singular, alternating, bilinear form. Let $\mathrm{Sp}_{2g}(F)$ denote the isometry group of $-\cdot -$. We can identify $\mathrm{Sp}_{2g}(\F_2)$ with the fixed subgroup of a Frobenius map on $\mathrm{Sp}_{2g}(F)$; see \cite[Ch.1.3.]{humphreys_modular_reps}

The following \textit{contraction maps} $\delta_k$ give invariant subspaces of $\bigwedge^k V_g$. For $k \geq 2$, they are given by
\begin{align} \label{contraction_maps_definition}
    \begin{split}
    \delta_k : \bigwedge^k V_g \rightarrow \bigwedge^{k-2} V_g \\
    w_1\wedge \cdots \wedge w_k \mapsto \sum_{i<j}(-1)^{i+j-1}(w_i \cdot w_j) w_1 \wedge \cdots \wedge \hat{w_i} \wedge \cdots \wedge \hat{w_j} \wedge \cdots \wedge w_k,
    \end{split}
\end{align}
and $\delta_0:\bigwedge^0 V_g \rightarrow 0, \delta_1: V_g \rightarrow 0$. These maps are $\mathrm{Sp}_{2g}(F)$--equivariant, so the $\ker(\delta)$ and $\Img(\delta)$ give invariant subspaces.

Since $F$ has characteristic two, the $\delta: \bigwedge V_g \rightarrow \bigwedge V_g$ satisfy $\delta^2=0$. Suppose $\{X_1,X_2,..,X_{2g} \}$ is a Symplectic basis for $V_g$, so that $$X_{2i} \cdot X_{2j} = X_{2i-1}\cdot X_{2j-1} = 0$$ and $$X_{2i-1}\cdot X_{2j} = \delta_{ij}.$$ Let $I = \{ 1,..,2g\}$ denote the indices, and call $B_i = \{ 2i-1, 2i\}$ the blocks, for $1 \leq i \leq g$. For a subset $J = \{ i_1,..,i_k\}$ of $I$ with $i_1<i_2<\cdots < i_k$, set $$X_J = X_{i_1} \wedge \cdots \wedge X_{i_k}.$$ We have a concrete description of $\delta_k$ in this basis, with two cases:
\begin{itemize}
    \item If $J$ contains $r \geq 1$ blocks $B_{i_1},...,B_{i_r}$, then $$\delta_k(X_J) = \sum_{j=1}^r X_{J - B_{i_j}}.$$ \\
    \item If $J$ contains no blocks, then $\delta_k(X_J) = 0$.
\end{itemize}

Now consider $\bigwedge^2 V_g$, we get an invariant subspace spanned by 
\begin{align} \label{omega_def}
\omega = \sum_{i=1}^g X_{2i-1} \wedge X_{2i}.
\end{align} Note that the kernel of $\delta_2: \bigwedge^2 V_g \rightarrow F$ is an invariant subspace, and $$\delta_2(\omega) = g = \begin{cases}
    0 \; , \; g \; \mathrm{even} \\
    1 \; , \; g \; \mathrm{odd}.
\end{cases}$$
Hence, when $g$ is odd, $\omega$ gives a section of $\delta_2$, so $$\bigwedge^2 V_g = F \oplus \ker(\delta_2).$$ When $g$ is even, we have a chain of invariant subspaces $$\langle \omega \rangle \subset \ker(\delta_2) \subset \bigwedge^2 V_g.$$

\subsection{Second differential} \label{second_diff_subsection}
Now we show that $\ker(\delta_2)$ is irreducible when $g$ is odd, and that $\ker( \delta_2)/\langle \omega \rangle$ is irreducible when $g$ is even. We do this by showing that the $\mathrm{Sp}_{2g}(\F_2)$--submodule generated by any non--zero element is the whole module.

Fix a symplectic basis $\{X_i\}_{i=1}^{2g}$, with $X_{2i-1} \cdot X_{2j} = \delta_{ij}$. Let $B_i = \{2i-1,2i\}$ denote the blocks of the index set $I = \{1,...,2g\}$. Define a subset $J \subset I$ to be isotropic if $J$ contains no blocks. Then the $X_J$ for $J \subset I$ isotropic, along with the $X_{B_1} + X_{B_i}$ for $2 \leq i \leq g$, form a basis for $\ker(\delta_2)$. We examine arbitrary $v \in \ker (\delta_2)$ using this basis, along with a set of elements in $\mathrm{Sp}_{2g}(F)$ discussed below.

Write a symplectic basis in the form $\{a_i,b_i\}_{i=1}^g$, where $$a_i \cdot b_j = \delta_{ij},\; a_i \cdot a_j = b_i \cdot b_j = 0.$$ The following generators of $\mathrm{Sp}_{2g}(\F_2)$ were described by Burkhardt; see \cite[p.164]{fm}.
\begin{itemize}
    \item Transvections: $a_1 \mapsto a_1 + b_1, b_1 \mapsto b_1, a_2 \mapsto a_2, b_2 \mapsto b_2.$
    \item Factor rotations: $a_1 \mapsto b_1, b_1 \mapsto a_1, a_2 \mapsto a_2, b_2 \mapsto b_2.$
    \item Factor mix: $a_1 \mapsto a_1+b_2, b_1 \mapsto b_1, a_2 \mapsto a_2-b_1, b_2 \mapsto b_2.$
    \item Factor swaps: for $1 \leq i \leq g$, consider the factor swap $a_i \mapsto a_{i+1}, b_i \mapsto b_{i+1}, a_{i+1} \mapsto a_i, b_{i+1} \mapsto b_i$.
\end{itemize}

\begin{lemma} \label{delta2_irred_part1}
    Let $\{ X_i \}_{i=1}^{2g}$ be a symplectic basis for $V_g$, with $X_{2i-1} \cdot X_{2j} = \delta_{ij}$. If $X_i \wedge X_j$ and $X_k \wedge X_l$ satisfy $X_i \cdot X_j = X_k \cdot X_l = 0$, then there exists $A \in \mathrm{Sp}_{2g}(\F_2)$ such that $A(X_i \wedge X_j) = X_k \wedge X_l$.
\end{lemma}
\begin{proof}
    Recall that a block is a subset of the form $B_r = \{2r-1,2r\}$ in the index set. We split the proof into cases.

    Case i: Suppose $X_i = X_k$ and that $X_j,X_l$ come from different blocks, so $j \in B, l \in B'$ and $i \notin B,B'$. Let $A$ denote the factor swap exchanging the blocks $B$ and $B'$, then $A(X_i \wedge X_j) = X_i \wedge A(X_j)$; if $A(X_j) = X_l$, we are done, otherwise postcompose with a factor rotation.

    Case ii: Suppose $X_i = X_k$ and $X_j,X_l$ come from the same block $B$. Let $A$ be the factor rotation for block $B$, then $A(X_i \wedge X_j) = X_i \wedge X_l = X_k \wedge X_l$.

    Case iii: Suppose there are two blocks in the set $\Lambda = \{ i,j,k,l\}$, then assume $\{j,l\}$ and $\{i,k\}$ are blocks. The map $$T: X_j \mapsto X_k, X_l \mapsto X_i, X_k \mapsto X_j, X_i \mapsto X_l$$
    is a composition of factor rotations and swaps, hence $T \in \mathrm{Sp}_{2g}(\F_2)$. Furthermore $T(X_i \wedge X_j ) = X_k \wedge X_l$.

    Case iv: Suppose there is one block in the set $\Lambda  = \{i,j,k,l \}$, so assume $B = \{ i,k\}$ is a block, $j \in B'$, $l \in B''$ with $B,B',B''$ distinct blocks. The factor swap $B' \leftrightarrow B''$ sends $X_i \wedge X_j$ to $X_i \wedge X_l$, and the factor rotation in $B$ sends $X_i \wedge X_l$ to $X_k \wedge X_l$.

    Case v: Suppose there are no blocks in the set $\Lambda  = \{ i,j,k,l\}$. Then $i \in B, j \in B', k \in B'', l \in B'''$ give four distinct blocks. The factor swap $B \leftrightarrow B''$ sends $X_i \wedge X_j$ to $X_k \wedge X_j$, and the factor swap $B' \leftrightarrow B'''$ sends $X_k \wedge X_j$ to $X_k \wedge X_l$.
\end{proof}

\begin{lemma}\label{delta2_irred_part2}
    The $\mathrm{Sp}_{2g}(\F_2)$--submodule generated by $X_1 \wedge X_3$ is $\ker(\delta_2)$.
\end{lemma}
\begin{proof}
    Let $T$ denote the transvection $X_1 \mapsto X_1 + X_2$, where other basis vectors are fixed. Then $T(X_1 \wedge X_3) = X_1 \wedge X_3 + X_2 \wedge X_3$. Let $A$ denote the factor mix $X_1 \mapsto X_1-X_4, X_2 \mapsto X_2, X_3 \mapsto X_3-X_2,$ with other basis vectors fixed. Then 
    \begin{align*}
    A(X_1 \wedge X_3 + X_2 \wedge X_3) = (X_1-X_4)\wedge (X_3-X_2) + X_2 \wedge (X_3-X_2) \\
    = X_1 \wedge X_3 + X_1 \wedge X_2 + X_4 \wedge X_3 + X_4 \wedge X_2 + X_2 \wedge X_3.
    \end{align*}
    Then Lemma \ref{delta2_irred_part1} implies that $X_1 \wedge X_2 + X_3 \wedge X_4$ is in the submodule generated by $X_1 \wedge X_3$. Then using factor swaps, we see that all $X_{B_1} + X_{B_i}$ are in the submodule generated by $X_1 \wedge X_3$. Combining this with Lemma \ref{delta2_irred_part1} proves the claim.
\end{proof}

\begin{lemma} \label{delta2_irred_part3}
    Let $v = \sum_{J \subset I \; \mathrm{isotropic}} a_JX_J + \sum_{i=2}^g b_i (X_{B_1} + X_{B_i}) \in \ker(\delta_2)$, then for all $1 \leq i \leq g$ the vector $$\sum_{j \neq 2i,2i-1}a_{\{2i-1,j\}}X_{\{2i-1,j\}}$$ is in the $\mathrm{Sp}_{2g}(\F_2)$--submodule generated by $v$, as well as the vector $$\sum_{j \neq 2i,2i-1}a_{\{2i,j\}}X_{\{2i,j\}}.$$ 
\end{lemma}
\begin{proof}
    Note that, for any basis vector $X_j$ and block $B_k$, we have $\tau_{X_j}(X_{B_k}) = X_{B_k}$, where $\tau_u$ denotes the transvection associated to $v$, defined by $\tau_u(w) = w + (w \cdot u)u$. Hence $$\tau_{X_{2i}}(v) = \sum_{J \subset I\; \mathrm{isotropic}}a_J\tau_{X_{2i}}(X_J) + \sum_{i}b_i(X_{B_1} + X_{B_i}).$$

    If $2i-1 \notin J$, then $\tau_{X_{2i}}(X_J) =X_J$, hence $$v - \tau_{X_{2i}}(v) = \sum_{\{2i-1,j\} \; j \neq 2i,2i-1}a_{\{2i-1,j\}}X_{2i} \wedge X_j := \beta.$$ We have $$\tau_{X_{2i-1}}(\beta) = \sum_{j \neq 2i-1,2i}a_{\{ 2i-1,j \}}(X_{2i}+X_{2i-1})\wedge X_j$$ and hence $$\beta - \tau_{X_{2i-1}}(\beta) = \sum_{j \neq 2i,2i-1}a_{\{ 2i-1,j \}}X_{2i-1} \wedge X_j.$$ Since $\beta - \tau_{X_{2i-1}}(\beta) = v - \tau_{X_{2i}}(v) - \tau_{X_{2i-1}}(v - \tau_{X_{2i}}(v))$, the first claim follows. A similar argument with $v - \tau_{X_{2i-1}}(v) - \tau_{X_{2i}}(v - \tau_{X_{2i-1}}(v))$ proves the second claim.
\end{proof}

\begin{lemma} \label{delta2_irred_part4}
    If $v = \sum_{J \subset I \; \mathrm{isotropic}} a_J X_J + \sum_i b_i(X_{B_1} + X_{B_i}) \in \ker (\delta_2)$ has at least one non--zero $a_J$ term, then the $\mathrm{Sp}_{2g}(\F_2)$--submodule generated by $v$ is $\ker (\delta_2)$.
\end{lemma}
\begin{proof}
    Using Lemma \ref{delta2_irred_part3}, we can assume $$u = \sum_{j \neq 2i,2i-1}a_{\{2i,j\}}X_{\{2i,j\}}$$ is in the submodule generated by $v$. Suppose there is an index $k$ with $a_{\{2i,2k-1\}} \neq 0$, then $$\tau_{X_{2k}}(u) - u = a_{\{ 2i ,2k-1\}}X_{\{ 2i, 2k\}},$$hence $X_{\{ 2i,2k\}}$ is in the submodule generated by $u$. The result then follows from Lemmas \ref{delta2_irred_part1} and \ref{delta2_irred_part2}. 
\end{proof}
We are left with the case where $v$ contains no vectors of the form $X_J$, for $J \subset I$ isotropic, as a summand.
\begin{lemma} \label{delta2_irred_part5}
    Suppose $v = \sum_i b_i (X_{B_1}+X_{B_i})$ and $v \notin \langle \omega \rangle$, then the $\mathrm{Sp}_{2g}(\F_2)$--submodule generated by $v$ is $\ker (\delta_2)$.
\end{lemma}
\begin{proof}
    Suppose $b_i \neq 0$. Consider the factor mix $T_{ij}: X_{2i-1} \mapsto X_{2i-1}+X_{2j}, X_{2i} \mapsto X_{2i}, X_{2j-1} \mapsto X_{2j-1} + X_{2i}, X_{2j} \mapsto X_{2j}$, where all other basis vectors are fixed. Then we have $$T_{ij}(v) - v = (b_i+b_j)X_{2j} \wedge X_{2i}.$$ If $b_i \neq b_j$ for some $j$, then $X_{2j} \wedge X_{2i}$ is in the submodule generated by $v$. The result then follows from Lemmas \ref{delta2_irred_part1} and \ref{delta2_irred_part2}. If $b_i = b_j$ for all $j$, then $v = \alpha \sum_{i=2}^g(X_{B_1} + X_{B_i}) = \begin{cases}
        \alpha \omega \;,\; g \; \mathrm{even} \\
        \alpha \sum_{i=2}^g X_{B_i} \; , \; g \; \mathrm{odd}.
    \end{cases}$ But when $g$ is odd, we have $$T_{12}(v) + v = \alpha X_2 \wedge X_4,$$ so $X_2 \wedge X_4$ is in the submodule generated by $v$, so the result follows from Lemmas \ref{delta2_irred_part1} and \ref{delta2_irred_part2}.
\end{proof}

\begin{thm} \label{ker_delta2_irred}
    For the contraction map $\delta_2:\bigwedge^2 V_g \rightarrow F$ defined in (\ref{contraction_maps_definition}), we have that $\ker(\delta_2)$ is irreducible as a $\mathrm{Sp}_{2g}(\F_2)$ representation when $g$ is odd, and that $\ker( \delta_2)/\langle \omega \rangle$ is irreducible as a $\mathrm{Sp}_{2g}(\F_2)$--representation when $g$ is even.
\end{thm}
\begin{proof}
    Let $v = \sum_{J \subset I \; \mathrm{isotropic}} a_J X_J + \sum_i b_i(X_{B_1} + X_{B_i}) \in \ker (\delta_2)$. If at least one $a_J$ term is nonzero, then by Lemma \ref{delta2_irred_part4}, the submodule generated by $v$ is $\ker (\delta_2)$. If all $a_J$ terms are zero, and $v \notin \langle \omega \rangle$, then by Lemma \ref{delta2_irred_part5}, the submodule generated by $v$ is $\ker(\delta_2)$. 

    Since $$ \sum_{i=1}^{g}(X_{B_1} + X_{B_i}) = gX_{B_1} + \sum_{i=1}^g X_{B_i}$$ we see that $$\omega = \sum_{i=1}^g X_{B_i}$$ is in $\ker (\delta_2)$ if and only if $g$ is even. In both cases, the previous paragraph shows that the $\mathrm{Sp}_{2g}(\F_2)$--module obtained from $\ker (\delta_2)$ is simple, which is equivalent to the representation being irreducible.
\end{proof}

\subsection{Third differential}

Recall that $$\omega = \sum_{i=1}^{g}X_{B_i}$$ is an invariant vector under the action of $\mathrm{Sp}_{2g}(F)$. Define 
\begin{align} \label{epsilon_definition}
    \begin{split}
    \epsilon: V_g \rightarrow \bigwedge^3 V_g \\
    v \mapsto \omega \wedge v,
    \end{split}
\end{align}
where $\omega \in \wedge^2V_g$ is defined in (\ref{omega_def}).

\begin{lemma} \label{wedge3_composition_part2}
    The image of $\epsilon$ is a $\mathrm{Sp}_{2g}(F)$--invariant subspace of $\bigwedge^3 V_g$. Furthermore, $\Img{\epsilon}$ is isomorphic to $V_g$ as a $\mathrm{Sp}_{2g}(F)$--module. For the contraction map $\delta_3: \bigwedge^3 V_g \rightarrow V_g$, we have $\Img{\epsilon} \subset \ker (\delta_3)$ if and only if $g$ is odd.
\end{lemma}
\begin{proof}
    Let $A \in \mathrm{Sp}_{2g}(F)$, and let $v \in V_g$, then $$\epsilon(Av) = \omega \wedge Av = A\omega \wedge Av = A(\omega \wedge v) = A(\epsilon(v)).$$ The map
    \begin{align*}
        \Psi: \Img{\epsilon} \rightarrow V_g \\
        \omega \wedge v \mapsto v
    \end{align*}
    is inverse to $\epsilon$, and is $\mathrm{Sp}_{2g}(F)$--equivariant by a similar calculation. We have $\delta_3(\omega \wedge X_{2k-1}) = (g-1)X_{2k-1}$, and $\delta_3(\omega \wedge X_{2k}) = (g-1)X_{2k}$, showing the last claim.
\end{proof}

Note that $\delta_3:\bigwedge^3 V_g \rightarrow V_g$ is surjective, hence $$\bigwedge^3 V_g / \ker (\delta_3) \cong V_g$$ as a $\mathrm{Sp}_{2g}(F)$--module. Hence $\bigwedge^3 V_g / \ker (\delta_3)$ is irreducible.

According to \cite[Thm.4.2]{gow_jlms}, when $g \geq 4$, we have $$\ker (\delta_3) = \Img{\delta_5}.$$ For $g$ odd, we obtain the filtration $$\Img{\epsilon} \subset \ker (\delta_3) \subset \bigwedge^3 V_g,$$ and for $g$ even, we obtain the filtration
$$\ker (\delta_3) \subset \bigwedge^3 V_g.$$ Next, we show that, depending on the parity of $g$, $\ker (\delta_3)$ or $\ker (\delta_3)/\Img{\epsilon}$ is irreducible, using a sequence of Lemmas leading up to Lemma \ref{kerd3_irreducible}. The first lemma is a computation.

\begin{lemma}\label{wedge3_composition_part4}
    Let $B_i = \{2i-1,2i\}$ be a block. Let 
    \begin{align*}
    v = \sum_{J \subset I \; \mathrm{isotropic}}a_J X_J + \sum_{l=1}^{2g} \sum_{l \notin B_i}a_{B_i \cup \{l\}}X_{B_i \cup \{ l \}} \in \ker (\delta_3).
    \end{align*}
    Let $\{k,j\}, \{t,r\}$ and $\{m,n\}$ be distinct blocks. Then $$\alpha = v - \tau_{X_k}(v) = \sum_{\{j,l,s\} \; \mathrm{isotropic}}a_{\{j,l,s\}}X_{\{k,l,s \}} + \sum_{j \notin B_i}a_{B_i \cup \{j\}}X_{B_i \cup \{k\}},$$
    $$\beta = \alpha - \tau_{X_t}(\alpha) = \sum_{\{ j,r,s \} \; \mathrm{isotropic}}a_{\{ j ,r, s \}}X_{\{ k, t, s \}},$$
    and 
    $$\gamma = \beta - \tau_{X_m}(\beta) = a_{\{j,r,n\}}X_{\{ k,t,m\}}.$$
\end{lemma}

\begin{lemma} \label{wedge3_composition_part7}
    Let $g \geq 3$, and let \begin{align*}
    v = \sum_{J \subset I \; \mathrm{isotropic}}a_J X_J + \sum_{l=1}^{2g} \sum_{l \notin B_i}a_{B_i \cup \{l\}}X_{B_i \cup \{ l \}} \in \ker (\delta_3)
    \end{align*} be as in Lemma \ref{wedge3_composition_part4}. Suppose at least one $a_J \neq 0$, for $J \subset I$ isotropic, then $X_J$ is in the $\mathrm{Sp}_{2g}(\F_2)$--submodule generated by $v$, for all $J \subset I$ isotropic.
\end{lemma}
\begin{proof}
    The $J \subset I$ isotropic with $a_J \neq 0$ specifies three blocks to use in Lemma \ref{wedge3_composition_part4}. Hence we get \begin{align*}
        a_{\{ j, r, n\}}X_{\{ k,t,m\}} = \beta - \tau_{X_m}(\beta) \\
        = (\alpha - \tau_{X_t}(\alpha)) -\tau_{X_m}(\alpha - \tau_{X_t}(\alpha)) \\
        = (v-\tau_{X_k}(v) - \tau_{X_t}(v - \tau_{X_k}(v))) - \tau_{X_m}(v-\tau_{X_k}(v) - \tau_{X_t}(v-\tau_{X_k}(v)))
    \end{align*}
    is in the submodule generated by $v$. Since $X_{\{ k,t,m\}}$ is in the submodule genrated by $v$, we obtain all $X_J$ for $J \subset I$ isotropic using factor swaps and factor rotations.
\end{proof}
Next we show that the submodule generated by any $X_J$ for $J \subset I$ isotropic is $\ker (\delta_3)$.
\begin{lemma}\label{wedge3_composition_part8}
    Let $g \geq 4$, then the $\mathrm{Sp}_{2g}(\F_2)$--submodule generated by $X_1 \wedge X_3 \wedge X_5$ is $\ker (\delta_3)$.
\end{lemma}
\begin{proof}
    According to \cite[Thm.4.2]{gow_jlms}, we have $\ker (\delta_3) = \Img{\delta_5}$. We compute $\delta_5(X_J)$ for $J \subset I$, $|J| =5$, to find a spanning set for $\ker (\delta_3)$. When $J$ contains one block $B$, we have $\delta_5(X_J) = X_{J-B}$, and the subset $J-B$ is isotropic. When $J$ contains two blocks $B,B'$ and another $j \notin B,B'$, we have $J = B \cup B' \cup \{ j \}$, so $\delta_5(X_J) = X_{B' \cup \{ j \}} + X_{B \cup \{ j \}}$. If $J \subset I$ contains no blocks, then $\delta_5(X_J) = 0$. Therefore, the $X_J$ for $J \subset I$ isotropic, along with the $X_{B \cup \{ j \}}+ X_{B' \cup \{j\}}$ for $j \notin B,B'$ and $B,B'$ blocks, span $\ker(\delta_3)$.

    Lemma \ref{wedge3_composition_part7} shows that all $X_J$ for $J \subset I$ isotropic is in the submodule generated by $X_{135}$. Note that, if one $X_{B \cup \{ j \}}+X_{B' \cup \{ j \}}$ is in the submodule generated by $X_{135}$, then so are all other vectors of this form, since we can apply factor swaps and factor rotations.

    Consider the factor mix $T: X_1 \mapsto X_1 + X_4,\; X_2 \mapsto X_2,\; X_3 \mapsto X_3 + X_2,\; X_4 \mapsto X_4$, with all other basis vectors fixed. We have $$T(X_1 \wedge X_3 \wedge X_5) = X_1\wedge X_3 \wedge X_5 + X_2 \wedge X_4 \wedge X_5 + (X_1 \wedge X_2 \wedge X_5 + X_3 \wedge X_4 \wedge X_5).$$ Since $X_2 \wedge X_4 \wedge X_5$ is isotropic, we get that $X_1 \wedge X_2 \wedge X_5 + X_3 \wedge X_4 \wedge X_5$ is in the submodule.
\end{proof}
\begin{thm}\label{kerd3_irreducible}
    Let $g \geq 4$. Let $\delta_3$ and $\epsilon$ be the operators defined in (\ref{contraction_maps_definition}) and (\ref{epsilon_definition}). If $g$ is odd then $\ker (\delta_3) / \Img{\epsilon}$ is irreducible as a $\mathrm{Sp}_{2g}(\F_2)$ representation. If $g$ is even, then $\ker (\delta_3)$ is irreducible as a $\mathrm{Sp}_{2g}(\F_2)$--representation.
\end{thm}
\begin{proof}
    Let
    \begin{align*}
    v = \sum_{J \subset I \; \mathrm{isotropic}}a_J X_J + \sum_{l=1}^{2g} \sum_{l \notin B_i}a_{B_i \cup \{l\}}X_{B_i \cup \{ l \}} \in \ker (\delta_3),
    \end{align*}
    if at least one $a_J \neq 0$ for $J \subset I$ isotropic, then the submodule generated by $v$ is $\ker(\delta_3)$ by Lemmas \ref{wedge3_composition_part7} and \ref{wedge3_composition_part8}. So suppose $a_J = 0$ for all $J \subset I$ isotropic. 
    
    By the proof of Lemma \ref{wedge3_composition_part8}, we can write $$v = \sum_{(B,B',l)} a_{(B,B',l)}(X_{B \cup \{ l \}} + X_{B' \cup \{ l \}}) = \sum_l (\sum_{l \notin B,B'} a_{(B,B',l)}(X_B + X_{B'}))\wedge X_l,$$
    where $(B,B',l)$ denotes a pair of distinct blocks $B,B'$ with $l \notin B,B'$. Fix a block $\{j,k\}$, with $\sum_{j \notin B,B'} a_{(B,B'j)}(X_{B \cup \{j\}}+X_{B' \cup \{ j\}}) \neq 0$. For the transvection $\tau_{X_k}$, we have $$v - \tau_{X_k}(v) = \sum_{j \notin B,B'} a_{(B,B',j)}(X_{B \cup \{ k \}}+X_{B' \cup \{ k \}}) = (\sum_{j \notin B,B'}a_{(B,B',j)}(X_B + X_{B'})) \wedge X_k.$$

    Since the block $\{j,k\}$ doesn't appear in the sum $\sum_{j \notin B,B'}a_{(B,B',j)}(X_B +X_{B'})$, we interpret it as an element of $\ker(\delta_2: \wedge^2 V_{g-1} \rightarrow F)$. Consider $\mathrm{Sp}_{2(g-1)}(\F_2)$ as a subgroup of $\mathrm{Sp}_{2g}(F)$ fixing the basis vectors $X_j$ and $X_k$. Suppose that $\sum_{j \notin B,B'}a_{(B,B',j)}(X_B+X_{B'}) \notin \langle \omega \rangle$, where $\omega = \sum_{B_i \neq \{ j,k\}}X_{B_i}$, then by Theorem \ref{ker_delta2_irred}, there exists $\alpha_l \in F$ and $g_l \in \mathrm{Sp}_{2(g-1)}(\F_2)$ such that $$\sum_l \alpha_l g_l(\sum_{j \notin B,B'}a_{(B,B',j)}(X_B+X_{B'})) = X_J$$ for some $J \subset I$ isotropic with $j,k \notin J$. But then \begin{align*}
        \sum_l \alpha_l g_l ((\sum_{j \notin B,B'}a_{(B,B',j)}(X_B + X_{B'})) \wedge X_k) = (\sum_l \alpha_l g_l (\sum_{j \notin B,B'}a_{(B,B',j)}(X_B + X_{B'}))) \wedge X_k \\
        = X_J \wedge X_k,
    \end{align*}
    with $J \cup \{ k\}$ isotropic, hence submodule generated by $v$ is $\ker(\delta_3)$ by Lemmas \ref{wedge3_composition_part7} and \ref{wedge3_composition_part8}.

    Finally, note that if $v \notin \Img (\epsilon)$, then the assumption that $\sum_{j \notin B,B'}a_{(B,B',j)}(X_B + X_{B'}) \notin \langle \omega \rangle$ for some block $\{j,k\}$ holds; this follows from proving the contrapositive statement.
\end{proof}

\subsection{Genus three case} \label{exceptional_composition_factors_subsection}

To study $\bigwedge^3 H_1(\Sigma_3)$, we follow \cite{gow_exterior_spin}. The codimension of $\Img (\delta_5)$ in $\ker (\delta_3)$ is $8$ by \cite[Thm.2.6]{gow_exterior_spin}. We describe the composition factors in terms of a symplectic basis $\{ X_i \}_{i=1}^{6}$ satisfying $X_{2i-1} \cdot X_{2j} = \delta_{ij}$:

Consider the subspace of dimension $8$ spanned by all vectors $X_J$ with $J \subset I$ isotropic. This subspace lies in $\ker (\delta_3)$ and has trivial intersection with $\Img (\delta_5)$, so the $X_J + \Img (\delta_5)$, $J$ as above, forms a basis for $\ker (\delta_3) / \Img (\delta_5)$. We obtain an irreducible $\mathrm{Sp}_6(\F_2)$--module $\ker (\delta_3) / \Img (\delta_5)$, referred to as the Spin module in \cite[Sect.3]{gow_exterior_spin}.

The image of $\delta_5$ under the basis $$\{ X_{B_1 \cup B_2 \cup \{5\}}, X_{B_1 \cup B_2 \cup \{6\}}, X_{B_1 \cup B_3 \cup \{ 3\}}, X_{B_1 \cup B_3 \cup \{4\}}, X_{B_2 \cup B_3 \cup \{ 1\}}, X_{B_2 \cup B_3 \cup \{2\}} \}$$ is the span of six vectors, where each vector is a sum of two distinct basis vectors, hence $\delta_5$ is injective, so $\Img{\delta_5} \cong \bigwedge^5 H_1(\Sigma_3) \cong H_1(\Sigma_3)^* \cong H_1(\Sigma_3)$ as a $\mathrm{Sp}_6(\F_2)$--module.

We have $\bigwedge^3 H_1(\Sigma_3) / \ker(\delta_3) \cong H_1(\Sigma_3)$ as a $\mathrm{Sp}_6(\F_2)$--module, so the filtration
$$0 \subset \Img{\delta_5} \subset \ker (\delta_3) \subset \bigwedge^3 H_1(\Sigma_3)$$
is a composition series.

Theorem \ref{ker_delta2_irred} gives $$\bigwedge^2 H_1(\Sigma_3) = \F_2 \oplus \ker (\delta_2)$$ as a $\mathrm{Sp}_6(\F_2)$--module, with $\ker(\delta_2)$ is irreducible. In summary, the $\mathrm{Sp}_6(\F_2)$--module $W_{\sigma, 3} \otimes_{\Z} \Z/2$ has composition factors $$H_1(\Sigma_3), \ker{\delta_3}/\Img{\delta_5}, H_1(\Sigma_3), \F_2, \ker(\delta_2), H_1(\Sigma_3).$$

\subsection{Summary}
\begin{theorem}\label{comp_factors_ab_level2mod_field_char2}
    Let $F$ be a field of characteristic $2$, considered as a trivial $\Mod_{g,1}[2]$--module. The $\mathrm{Sp}_{2g}(\F_2)$--module $H_1(\Mod_{g,1}[2];F)$ has the following table of composition factors, where $\delta_k$ and $\epsilon$ are defined in (\ref{contraction_maps_definition}) and (\ref{epsilon_definition}) using the symplectic vector space $(H_1(\Sigma_g;F),Q)$, and $\omega$ is given by (\ref{omega_def}). The action of $\mathrm{Sp}_{2g}(\F_2) = \Mod_{g,1}/\Mod_{g,1}[2]$ on $H_1(\Sigma_g;F)$ is induced from the action of $\Mod_{g,1}$ on $H_1(\Sigma_g;F)$.
    \begin{table}[h]
\begin{tabular}{|l|l|l|}
\hline
Genus                        & Composition factors                                                                                                                                                                                                                         & Multiplicities \\ \hline

$g \geq 4$ odd  & $H_1(\Sigma_g;F), F, \ker(\delta_2), \ker(\delta_3) / \Img{\epsilon}$                      & $3,1,1,1$        \\ \hline
$g \geq 4$ even & $H_1(\Sigma_g;F), F, \ker(\delta_2) / \langle \omega \rangle, \ker(\delta_3)$ & $2,2,1,1$        \\ \hline
\end{tabular}
\end{table}
\end{theorem}
\begin{proof}
    We have a field extension $\F_2 =\{ 0,1\} \subset F$. We have that $H_1(\Mod_{g,1}[2];\Z)$ is isomorphic to $W_{\sigma,g}$ as a $\mathrm{Sp}_{2g}(\F_2)$--module. The universal coefficient theorem and associativiy of tensor products give the following sequence of $\mathrm{Sp}_{2g}(\F_2)$--module isomorphisms $$H_1(\Mod_{g,1}[2];F) \cong H_1(\Mod_{g,1}[2];\Z) \otimes_{\Z} F \cong (H_1(\Mod_{g,1}[2];\Z)\otimes_{\Z} \F_2) \otimes_{\F_2} F \cong (W_{\sigma,g} \otimes_{\Z} \F_2) \otimes_{\F_2}F.$$ 
    
    Since $- \otimes_{\F_2} F$ is exact, Lemma \ref{level2ab_filtration_submod} gives the following filtration $$0 \subset Z_g \otimes_{\F_2} F \subset Q_g \otimes_{\F_2} F \subset (W_{\sigma, g} \otimes_{\Z} \Z/2)\otimes_{\F_2} F$$
    by $\mathrm{Sp}_{2g}(\F_2)$--submodules. We obtain $Z_g \otimes_{\F_2} F \cong (\wedge^3 H_1(\Sigma_g))\otimes_{\F_2} F = \wedge^3 (H_1(\Sigma_g) \otimes_{\F_2} F)$, $Q_g \otimes_{\F_2} F / Z_g \otimes_{\F_2} F \cong \wedge^2 (H_1(\Sigma_g) \otimes_{\F_2}F)$, and $(W_{\sigma, g} \otimes_{\Z} \Z/2)\otimes_{\F_2} F / Q_g \otimes_{\F_2} F \cong H_1(\Sigma_g) \otimes_{\F_2} F$ as $\mathrm{Sp}_{2g}(\F_2)$--modules. Furthermore $H_1(\Sigma_g) \otimes_{\F_2}F$ is isomorphic to $H_1(\Sigma_g;F)$ as $\mathrm{Sp}_{2g}(\F_2)$--modules, by naturality of the universal coefficient theorem.

    The result now follows from Theorems \ref{ker_delta2_irred} and \ref{kerd3_irreducible}.
\end{proof}

\section{Action on the image of the Birman--Craggs--Johnson map} \label{bcj_section}
Let $\mathcal{I}_{g,1}$ denote the Torelli subgroup of the mapping class group $\Mod_{g,1}$, that is, the kernel of the action on $H_1(\Sigma_g;\Z)$. We obtain an action of $\Mod_{g,1}$ on $H_1(\mathcal{I}_{g,1};\Z)$ induced by the conjugation action on $\mathcal{I}_{g,1}$.

Let $H_1(\mathcal{I}_{g,1};\Z)_{\Mod_{g,1}[2]}$ denote the module of coinvariants, that is, the largest quotient of $H_1(\mathcal{I}_{g,1};\Z)$ on which $\Mod_{g,1}[2]$ acts trivially. This is a $\Mod_{g,1}/\Mod_{g,1}[2] = \mathrm{Sp}_{2g}(\F_2)$--module. Johnson showed that $H_1(\mathcal{I}_{g,1};\Z)_{\Mod_{g,1}[2]}$ is isomorphic to $B^3_{g,1}$ as a $\mathrm{Sp}_{2g}(\F_2)$--module, where $B^3_{g,1}$ is defined below; for more details see \cite[Thm.4 and 6, Lem.13]{BCJpaper}, \cite[Sect.2 and Thm.1,2]{johnsonab}, and \cite[Sect.7]{sato}.

To define $B^3_{g,1}$, let $R$ denote the commutative polynomial ring in the variables $\bar x$, where $x \in H_1(\Sigma_g ; \F_2)$, with coefficients in $\F_2$. Let $J$ be the ideal of $R$ generated by all $$\overline{x+y} -(\overline{x} + \overline{y} + x \cdot y),\; \overline{x}^2-\overline{x}$$
where $x,y \in H_1(\Sigma_g;\F_2)$. Let $R_3$ denote the submodule of $R$ consisting of all polynomials of degree at most three, and let $$B^3_{g,1} = R_3 / J \cap R_3.$$ We obtain an $\F_2$--linear action of $\mathrm{Sp}_{2g}(\F_2)$ on $B^3_{g,1}$, given on monomials by the formula $$h \cdot \overline{x} = \overline{hx}$$ for $x \in H_1(\Sigma_g;\F_2)$ and $h \in \mathrm{Sp}_{2g}(\F_2)$. 

\begin{lemma} \label{torelli_coinvariants_filtration} \cite[Prop.4]{johnsonab}
    Let $K_g$ denote the span of $1$ in $B^3_{g,1}$. Let $L_g$ denote the span of all polynomials of the form $\overline{a}$ in $B^3_{g,1}$, where $a \in H_1(\Sigma_g;\F_2)$. Let $Q_g$ denote the span of all polynomials of the form $\overline{a} \overline{b}$ in $B^3_{g,1}$, where $a,b \in H_1(\Sigma_g;\F_2)$. We have a filtration of $\mathrm{Sp}_{2g}(\F_2)$--submodules 
    \begin{align} \label{torelli_submod_filtration}
        K_g \subset L_g \subset Q_g \subset B^3_{g,1}.
    \end{align}
    Furthermore, we have that $L_g/K_g \cong H_1(\Sigma_g;\F_2),\; Q_g/L_g \cong \wedge^2 H_1(\Sigma_g;\F_2)$, and $B^3_{g,1}/Q_g \cong \wedge^3 H_1(\Sigma_g;\F_2)$ as $\mathrm{Sp}_{2g}(\F_2)$--modules.
\end{lemma}

Using this filtration, we obtain the Jordan--Holder decomposition of $H_1(\mathcal{I}_{g,1};\Z)_{\Mod_{g,1}[2]}$.

\begin{theorem} \label{jordan_holder_torelli_coinvariants}
    Let $F$ be a field of characteristic $2$, regarded as a trivial $\Mod_{g,1}[2]$--module. Let $g \geq 3$. If $g$ is even, then the composition factors of the $\mathrm{Sp}_{2g}(\F_2)$--module $H_1(\mathcal{I}_{g,1};F)_{\Mod_{g,1}[2]}$ are $F, H_1(\Sigma_g;F), \ker(\delta_2)/\langle \omega \rangle, \ker(\delta_3)$ with multiplicities $3,2,1,1$ respectively. If $g$ is odd, then the composition factors are $F, H_1(\Sigma_g;F), \ker(\delta_2), \ker(\delta_3)/\Img (\epsilon)$ with multiplicities $2,2,1,1$ respectively. Here, the action of $\Mod_{g,1}$ on $H_1(\Sigma_g;F)$ descends to an action of $\Mod_{g,1}/\Mod_{g,1}[2] = \mathrm{Sp}_{2g}(\F_2)$, and $\delta_k, \omega$, and $\epsilon$ are defined in (\ref{contraction_maps_definition}),(\ref{omega_def}), and (\ref{epsilon_definition}), using the symplectic space $(H_1(\Sigma_g;F),Q)$.
\end{theorem}
\begin{proof}
    For $F = \F_2$, this follows from the filtration in Lemma \ref{torelli_coinvariants_filtration}, and Theorems \ref{kerd3_irreducible} and \ref{ker_delta2_irred}.

    By the universal coefficient theorem, we obtain a natural isomorphism $H_1(\mathcal{I}_{g,1};F) \cong H_1(\mathcal{I}_{g,1};\Z) \otimes_{\Z} F$. There is a natural isomorphism $$(H_1(\mathcal{I}_{g,1};\Z) \otimes_{\Z} F)_{\Mod_{g,1}[2]} \cong H_1(\mathcal{I}_{g,1};\Z)_{\Mod_{g,1}[2]} \otimes_{\Z} F.$$ Since $H_1(\mathcal{I}_{g,1};\Z)_{\Mod_{g,1}[2]}$ is an $\F_2$--module, using the exactness of $- \otimes_{\F_2} F$, the filtration of Lemma \ref{torelli_coinvariants_filtration}, and an argument similar to Theorem \ref{comp_factors_ab_level2mod_field_char2} gives the result.
\end{proof}

\section{Representation stability} \label{rep_stab_sp_section}
Suppose $F = \F_2$ or its algebraic closure $\overline{\F}_2$. We can label the simple modules that appear for $H_1(\Mod_{g,1}[2];F)$ and $H_1(\mathcal{I}_{g,1};F)_{\Mod_{g,1}[2]}$ by using the theory of highest weights for the algebraic group $\Sp_{2g}(\overline{\F}_2)$. We recall this theory below, and use it to prove that both modules satisfy a form of representation stability. For more details on the highest weight theory, see \cite[Ch.2]{humphreys_modular_reps}, \cite[Sect.8.2]{church_farb}.

Let $V = \overline{\F}_2^{2g}$, and let $Q :V \times V \rightarrow \overline{\F}_2$ denote a non--singular alternating bilinear form on $V$. Let $$Sp(V) = \{ f \in GL(V) \; | \; Q(f  u, f  v) = Q(u,v) \; \forall u,v \in V\}.$$ Let $\{ a_i,b_i\}_{i=1}^g$ be a symplectic basis for $(V,Q)$, so that $Q(a_i,b_j) = \delta_{ij}$, with all other pairings zero. For $z_1,...,z_g \in \overline{\F}_2^*$, let $t(z_1,..,z_g) \in Sp(V)$ denote the element given by the assignment $$a_i \mapsto z_i a_i, \; b_i \mapsto z_i^{-1}b_i \; (1 \leq i \leq g).$$
Let $T = \{ t(z_1,..,z_g) \; | \; z_1,..,z_g \in \overline{\F}_2^*\}$. Then $T$ is a maximal torus for $Sp(V)$. Let $X(T) = \Hom (T,\mathbb{G}_m)$ denote the character group of $T$. Identify $X(T)$ with $\Z[L_1,..,L_g]$, where
\begin{align*}
    L_i: T \rightarrow \overline{\F}_2 \\
    t(z_1,..,z_g) \mapsto z_i,
\end{align*}
and we write multiplication of characters using additive notation. 

Let $$\omega_i = L_1+\cdots + L_i$$ denote the fundamental weights, and call a weight $\lambda \in X(T)$ dominant if it can be written as a nonnegative integral linear combination of the fundamental weights $\omega_i$; see \cite[Ch.31.1]{humphreys_algebraic_groups} for more details. The weights of the adjoint representation of $Sp(V)$ on its Lie algebra $sp(V)$ forms a root system $$\Phi = \{ \pm L_i \pm L_j \;|\; i \neq j\} \cup \{\pm 2L_i\},$$
where $sp(V)$ is the space of endomorphisms $A:V\rightarrow V$ such that $Q(Av,w) + Q(v,Aw) = 0$ for all $v,w \in V$. The choice of simple roots $\Delta$ and positive roots $\Phi^+$ given by $$\Delta = \{ L_i - L_{i+1} \; | 1 \leq i <g\} \cup \{ 2L_g\},\; \Phi^+ = \{ L_i \pm L_j \; |\; i<j\} \cup \{ 2L_i \;|\; 1 \leq i \leq g\}$$ determines a partial order on $X(T) \otimes_{\Z} \R$ given by $\lambda \leq \mu$ if $\mu - \lambda$ is a sum (possibly $0$) of positive roots.

According to \cite[Thm.2.2]{humphreys_modular_reps}, every (rational) irreducible representation of $\mathrm{Sp}_{2g}(\overline{\F}_2)$ over $\overline{\F}_2$ has a unique highest weight $\lambda$ under the partial ordering determined by the positive roots; furthermore, this highest weight $\lambda$ is dominant. If two irreducible representations have the same highest weight $\lambda$, then they are isomorphic, so we may unambiguously denote the representation by $L(\lambda)$.

A dominant weight $\lambda \in X(T)$ is $p$--restricted if $\lambda = \sum_i c_i \omega_i$ with $0 \leq c_i < p$. If $\lambda \in X(T)$ is a $p$--restricted weight, then the restriction of the $\mathrm{Sp}_{2g}(\overline{\F}_2)$ representation $L(\lambda)$ to $\mathrm{Sp}_{2g}(\F_2)$ remains irreducible, and every irreducible representation of $\mathrm{Sp}_{2g}(\F_2)$ over $\overline{\F}_2$ is obtained this way; see \cite[Theorems 2.2 and 2.11]{humphreys_modular_reps}. 

\begin{lemma} \label{2nd_fundamental_weight}
    Let $\delta_2$ and $\omega$ be defined as in (\ref{contraction_maps_definition}) and (\ref{omega_def}).
    If $g$ is odd, then $\ker \delta_2 = L(\omega_2)$. If $g$ is even, then $\ker \delta_2 / \langle \omega \rangle = L(\omega_2)$.
\end{lemma}
\begin{proof}
    If $g$ is odd, let $v = a_1 \wedge a_2$, and if $g$ is even, let $v = a_1 \wedge a_2 + \langle \omega \rangle$. Theorem \ref{ker_delta2_irred} shows that both representations considered are irreducible as $\mathrm{Sp}_{2g}(\overline{\F}_2)$--representations, and their restriction to $\mathrm{Sp}_{2g}(\F_2)$ remains irreducible. The vector $v$ is an eigenvector for the action of $T$ with weight $\omega_2 = L_1+L_2$. One can check that the weights of the $\mathrm{Sp}_{2g}(\overline{\F}_2)$--modules considered are $\{ \pm(L_i \pm L_j)\;|\; i \neq j\} \cup \{0\}$, and a computation shows that $\omega_2$ is the highest weight. 
\end{proof}

\begin{lemma} \label{3rd_fundamental_weight}
    Let $\delta_3$ and $\omega$ be defined as in (\ref{contraction_maps_definition}) and (\ref{omega_def}). If $g$ is even, then $\ker \delta_3=L(\omega_3)$. If $g$ is odd then $\ker \delta_3 / \Img \epsilon = L(\omega_3)$. 
\end{lemma}
\begin{proof}
    If $g$ is even, let $v = a_1 \wedge a_2 \wedge a_3$, and if $g$ is odd, let $v = a_1 \wedge a_2 \wedge a_3 + \Img \epsilon$. Theorem \ref{kerd3_irreducible} shows that both representations considered are irreducible $\mathrm{Sp}_{2g}(\overline{\F}_2)$--representations, and their restriction to $\mathrm{Sp}_{2g}(\F_2)$ remains irreducible. The vector $v$ is an eigenvector for the action of $T$ with weight $\omega_3 = L_1 + L_2 +L_3$. The weights of the $\mathrm{Sp}_{2g}(\overline{\F}_2)$--modules considered are $\pm (L_i+L_j+L_k),\; \pm(L_i - L_j + L_k)$ with $i,j,k$ distinct, along with the $\pm L_j$. A computation shows that $\omega_3$ is the highest weight.
\end{proof}

\begin{thm} \label{level2_stably_representation_periodic}
    Let $F=\F_2$ or its algebraic closure $\overline{\F}_2$. The $\mathrm{Sp}_{2g}(\F_2)$--representations $H_1(\Mod_{g,1}[2];F)$ and $H_1(\mathcal{I}_{g,1};F)_{\Mod_{g,1}[2]}$ are uniformly stably periodic in the sense of \cite[Def.8.1]{church_farb}.
\end{thm}
\begin{proof}
    For the sequence $\{ H_1(\Mod_{g,1}[2];F)\}$ we identify the $\Sp_{2g}(\F_2)$--module $H_1(\Mod_{g,1}[2];\Z)$ with the algebra $W_{\sigma,g}$ given in Lemma \ref{algebra_relations_lemma}, using Proposition \ref{Satomainresult}; see Section \ref{Sp_rep_section}. Following the notation in \cite[Def.8.1]{church_farb} we set $V_g = W_{\sigma,g} \otimes_{\Z} F$. The maps $\phi_g: V_g \rightarrow V_{g+1}$ are induced by the inclusion of $H_1(\Sigma_g;\F_2)$ in $H_1(\Sigma_{g+1};\F_2)$. Combining Theorems \ref{ker_delta2_irred}, \ref{kerd3_irreducible}, and Lemmas \ref{level2ab_filtration_submod}, \ref{2nd_fundamental_weight}, \ref{3rd_fundamental_weight}, we have the following: If $g \geq 4$ is odd, then the composition factors of $H_1(\Mod_{g,1}[2];F)$ are $L(\omega_1), L(0), L(\omega_2), L(\omega_3)$ with multiplicities $3,1,1,1$. If $g \geq 4$ is even, then the composition factors of $H_1(\Mod_{g,1}[2];F)$ are $L(\omega_1), L(0),L(\omega_2), L(\omega_3)$ with multiplicities $2,2,1,1$.

    The sequence of $\mathrm{Sp}_{2g}(\F_2)$--modules $\{ H_1(\mathcal{I}_{g,1};\Z)_{\Mod_{g,1}[2]}\}$ can be identified with the sequence $\{B^3_{g,1}\}$ described in Section \ref{bcj_section}. Then the maps $B^3_{g,1} \rightarrow B^3_{g+1,1}$ are induced by the inclusion of $H_1(\Sigma_g;\F_2)$ in $H_1(\Sigma_{g+1};\F_2)$. Combining Theorem \ref{jordan_holder_torelli_coinvariants} with Lemmas \ref{2nd_fundamental_weight} and \ref{3rd_fundamental_weight}, we have the following: If $g \geq 3$ is even, then the composition factors of $H_1(\mathcal{I}_{g,1};F)_{\Mod_{g,1}[2]}$ are $L(0), L(\omega_1), L(\omega_2), L(\omega_3)$ with multiplicities $3,2,1,1$. If $g \geq 3$ is odd, then the composition factors of $H_1(\mathcal{I}_{g,1};F)_{\Mod_{g,1}[2]}$ are $L(0), L(\omega_1), L(\omega_2), L(\omega_3)$ with multiplicities $2,2,1,1$.
\end{proof}

\section{Invariant subspaces}
\label{bp_sep_twist_eval_subsection}
Consider the images of normal subgroups $N$ of $\Mod_{g,1}[2]$ in $H_1(\Mod_{g,1};\F_2)$ under $\beta_{\sigma}$; the action of $\mathrm{Sp}_{2g}(\F_2)$ is given by the conjugation action of $\Mod_{g,1}$ on $\Mod_{g,1}[2]$, and since $N$ is normal, it's image in $W_{\sigma,g} \otimes_{\Z} \F_2$ is an invariant subspace. In this section, we study the image of the Torelli group and Johnson kernel under $\beta_{\sigma}$ in $W_{\sigma,g} \otimes_{\Z} \F_2$, using the relations of Lemma \ref{algebra_relations_lemma}. To do this, we use results about bounding pair maps, and separating curves.

\textbf{Factoring bounding pairs into squares.} The result under Proposition 4.12 of \autocite{fm} allows us to factor a genus $1$ bounding pair into a product of squares of Dehn twists. 

Define a chain of simple closed curves $c_1, c_2, c_3$ to be a triple such that $i(c_1,c_2)=i(c_2,c_3) = 1$ and all other pairwise geometric intersection numbers are zero. Let $c_1, c_2, c_3$ be a chain of simple closed curves and let $d_1,d_2$ be the boundary curves of a regular neighbourhood of $c_1 \cup c_2 \cup c_3$, then the chain relation gives
\begin{equation*}
    (t_{c_1}^2 t_{c_2} t_{c_3})^3 = t_{d_1} t_{d_2}.
\end{equation*}
So we have that 
\begin{align*}
    (t_{c_1}^2 t_{c_2} t_{c_3})^3 = t_{c_1}^2(t_{c_2} t_{c_3} t_{c_1}^2 t_{c_3}^{-1} t_{c_2}^{-1})(t_{c_2} t_{c_3})^2(t_{c_1}^2 t_{c_2} t_{c_3}) \\
    = t_{c_1}^2(t_{c_2} t_{c_3} t_{c_1}^2 t_{c_3}^{-1} t_{c_2}^{-1})(t_{c_2} t_{c_3})^2t_{c_1}^2(t_{c_2}t_{c_3})^{-2} (t_{c_2}t_{c_3})^3.
\end{align*}
Using the braid relation, the rightmost term in the last equality can be written as 
\begin{align*}
    (t_{c_2} t_{c_3})^3 = t_{c_3} t_{c_2} t_{c_3}^2 t_{c_2} t_{c_3} \\
    = (t_{c_3} t_{c_2} t_{c_3}^2 t_{c_2}^{-1} t_{c_3}^{-1}) (t_{c_3} t_{c_2}^2 t_{c_3}) \\
    = (t_{c_3} t_{c_2} t_{c_3}^2 t_{c_2}^{-1} t_{c_3}^{-1}) (t_{c_3} t_{c_2}^2 t_{c_3}^{-1}) t_{c_3}^2,
\end{align*}
and so we have the following factorisation
\begin{equation*}
    t_{d_1} t_{d_2}^{-1} = t_{c_1}^2 t_{t_{c_2}t_{c_3}(c_1)}^2 t_{(t_{c_2} t_{c_3})^2(c_1)}^2 t_{t_{c_3} t_{c_2}(c_3)}^2 t_{t_{c_3}(c_2)}^2 t_{c_3}^2 t_{d_2}^{-2}.
\end{equation*}

\textbf{Image of the Torelli group} To calculate the image of the Torelli group, we use Johnsons result that the maps $t_{d_1}t_{d_2}^{-1}$ running over all $d_1, d_2$ bounding pairs of genus $1$ generate $\mathcal{I}_{g,1}$ \autocite[Theorem 1]{johnsonbp_genus_1_generates}. If we evaluate $\beta_{\sigma}$ on one such bounding pair map, and use $\beta_{\sigma}(\varphi t_c^2 \varphi^{-1}) = \beta_{\sigma}(t_{\varphi(c)}^2) = (-1)^{q_{\sigma}(\varphi_*[c])}i_{\varphi_*[c]}$, we describe the image, since bounding pairs of genus $1$ are all in the same orbit under the conjugation action $\Mod_{g,1}$ on $\mathcal{I}_{g,1}$. 

The element $\beta_{\sigma}(t_c^2)$ only depends on the homology class $C \in H_1(\Sigma_g)$ of the curve $c$. Denote the homology class of the curve $c_i$ by $C_i$. Motivated by the formula for bounding pairs written above, we compute that in $H_1(\Sigma_g)$
\begin{equation*}
    t_{c_2}t_{c_3}(C_1) = C_1 + C_2,
\end{equation*}
\begin{equation*}
    (t_{c_2}t_{c_3})^2(C_1) = C_1 + C_2 + C_3,
\end{equation*}
\begin{equation*}
    t_{c_3}t_{c_2}(C_3) = C_2,
\end{equation*}
\begin{equation*}
    t_{c_3}(C_2) = C_2+C_3.
\end{equation*}
Note that $C_1, C_2, D_1 = D_2$ can be completed into a symplectic basis for $H_1(\Sigma_g)$. 

\begin{cor} \label{img_torelli_mod2}
    The image of the Torelli group $\mathcal{I}_{g,1}$ under $\beta_{\sigma}$ in $W_{\sigma, g} \otimes_{\Z} \F_2$ is the $\mathrm{Sp}_{2g}(\F_2)$--submodule generated by $(4 \overline{C_1} \; \overline{C_2} \; \overline{D_1}) \otimes 1$.
\end{cor}
\begin{proof}
    Substituting in the factorisation of a bounding pair, we obtain 
\begin{equation*}
    \beta_{\sigma}(t_{d_1}t_{d_2}^{-1}) = \overline{C_1} + \overline{C_1 + C_2} + \overline{C_1 + C_2 + C_3} + \overline{C_2} + \overline{C_2 + C_3} + \overline{C_3} - \overline{D_1}.
\end{equation*}
    Since $C_1+D_1=C_3$, applying Lemma \ref{algebra_relations_lemma} gives   $$\beta_{\sigma}(t_{d_1} t_{d_2}^{-1}) \otimes 1 = (4\overline{C_1} \; \overline{C_2} \; \overline{D_1})\otimes 1.$$
\end{proof}

\textbf{Image of the Johnson kernel.} The Johnson kernel $\mathcal{K}_{g,1} \subset \mathcal{I}_{g,1}$ is the subgroup of the mapping class group generated by all separating twists. Suppose we have a chain $c_1,..,c_k$ of simple closed curves in $\Sigma_{g,1}$, so $i(c_i,c_{i+1})=1$, and $i(c_i,c_j) = 0$ for $|i-j|>1$. When $k$ is even, the boundary of a regular neighbourhood of the union of the $c_i$ is a separating curve $d$. The following relation holds \autocite[Proposition 4.12]{fm}
\begin{align*}
    t_d = (t_{c_1}^2t_{c_2}\cdots t_{c_k})^{2k}.
\end{align*}

For $k=2$, we have
\begin{align*}
    (t_{c_1}^2t_{c_2})^4 = t_{c_1}^2(t_{c_2}t_{c_1}^2t_{c_2}^{-1})t_{c_2}^2(t_{c_1}^2t_{c_2})^2 \\
    = t_{c_1}^2t_{t_{c_2}(c_1)}^2(t_{c_2}^2t_{c_1}^2t_{c_2}^{-2})t_{c_2}^3(t_{c_1}^2t_{c_2}) \\
    = t_{c_1}^2t_{t_{c_2}(c_1)}^2t_{t_{c_2}^2(c_1)}^2t_{t_{c_2}^3(c_1)}^2t_{c_2}^4.
\end{align*}

For even $k$, set $t_i = t_{c_i}$, $f_{2k} = t_2\cdots t_k$, and $\prod_{j=1}^nx_j = x_1\cdots x_n$, then the previous computation generalises to
\begin{align} \label{sep_twist_almost_productsquares}
    t_d = (t_1^2f_{2k})^{2k} = (\prod_{j=0}^{2k-1}t^2_{f_{2k}^j(c_1)}) \cdot (f_{2k})^{2k}.
\end{align}

For the $k=2$ case, the $C_1,C_2 \in H_1(\Sigma_g)$ can be completed to a symplectic basis, and we have $t_2(C_1) = C_1+C_2$, $t_2^2(C_1) = C_1$, and $t_2^3(C_1)=C_1+C_2$. Therefore
\begin{align*}
    \beta_{\sigma}(t_d) = 2\overline{C_1} + 2\overline{C_1+C_2} + 2\overline{C_2},
\end{align*}
implying that $\beta_{\sigma}(t_d) \otimes 1 = 0 \in W_{\sigma, g} \otimes_{\Z} \F_2$. 

\textbf{Image of congruence subgroups.} 
Let $\rho: \Mod_{g,1} \rightarrow \mathrm{Sp}_{2g}(\Z)$ denote the symplectic representation. Then $\rho$ is surjective; see \cite[Thm.6.4]{fm}. The \textit{level $m$ congruence subgroup} $\mathrm{Sp}_{2g}(\Z)[m]$ of $\mathrm{Sp}_{2g}(\Z)$ is the kernel of the reduction homomorphism $\mathrm{Sp}_{2g}(\Z) \rightarrow \mathrm{Sp}_{2g}(\Z/m)$. The \textit{level $m$ congruence subgroup} $\Mod_{g,1}[m]$ is the preimage of $\mathrm{Sp}_{2g}(\Z)[m]$ under $\rho$. We have a short exact sequence
$$1 \rightarrow \mathcal{I}_{g,1} \rightarrow \Mod_{g,1}[m] \rightarrow \mathrm{Sp}_{2g}(\Z)[m] \rightarrow 1.$$

Mennicke showed that the congruence subgroup $\mathrm{Sp}_{2g}(\Z)[m]$ is generated by all $\rho(t_c^m)$, where $c$ runs over non--separating simple closed curves in $\Sigma_{g,1}$ \cite[p.128]{mennicke_congruence_gens}. Johnson showed that $\mathcal{I}_{g,1}$ is generated by bounding pair maps of genus one \cite{johnsonbp_genus_1_generates}. The short exact sequence above implies that $\Mod_{g,1}[m]$ is generated by bounding pair maps of genus one, and all $t_c^m$, where $c$ runs over non--separating simple closed curves.

The reduction map $\mathrm{Sp}_{2g}(\Z) \rightarrow \mathrm{Sp}_{2g}(\F_2)$ is the composition of the reductions $\mathrm{Sp}_{2g}(\Z) \rightarrow \mathrm{Sp}_{2g}(\Z/2m) \rightarrow \mathrm{Sp}_{2g}(\F_2)$, hence $\Mod_{g,1}[2m]$ is a subgroup of $\Mod_{g,1}[2]$ that is normal in $\Mod_{g,1}$.

The generating set described above gives the following.
\begin{cor}
    For $k \geq 1$, the image of $\Mod_{g,1}[4k]$ in $W_{\sigma,g} \otimes_{\Z} \F_2$ under $\beta_{\sigma}$ is the image of the Torelli group. For $k$ odd, the image of $\Mod_{g,1}[2k]$ under $\beta_{\sigma}$ is $W_{\sigma,g} \otimes_{\Z} \F_2$.
\end{cor}

\section{Level $2$ congruence subgroup for a once punctured or a closed surface} \label{oncepunctured_or_closed_level2_case}

Let $\Sigma_g$ denote a closed, oriented surface of genus $g$, and let $\Mod_g$ denote the mapping class group $\pi_0(\mathrm{Diff}^+\Sigma_g)$. Let $p$ denote a point in $\Sigma_g$, and let $\Mod_g^1$ denote the mapping class group $\pi_0(\Diff^+(\Sigma_g,p))$, where $\Diff^+(\Sigma_g,p)$ is the space of all orientation-preserving diffeomorphisms of $\Sigma_g$ that fix $p$. Let $\Mod_g[2]$ and $\Mod_g^1[2]$ denote the kernel of the actions of $\Mod_g$ and $\Mod_g^1$ on $H_1(\Sigma_g;\Z/2)$ respectively. In this section, we use the calculation in Theorem \ref{comp_factors_ab_level2mod_field_char2}, along with the Birman exact sequence, to deduce analogous results for the $\Sp_{2g}$--representations $H_1(\Mod_g[2];F)$ and $H_1(\Mod_g^1[2];F)$. We begin with the once punctured case.

\begin{lemma} \label{oncepunctured_level2_sprep}
    Let $F$ be a field of characteristic $2$, then $H_1(\Mod_g^1[2];F)$ is isomorphic to $H_1(\Mod_{g,1}[2];F)$ as a $\Sp_{2g}(\F_2)$--module. Hence, the composition factors for $H_1(\Mod_g^1[2];F)$ are given by the table in Theorem \ref{comp_factors_ab_level2mod_field_char2}.
\end{lemma}
\begin{proof}
Let $\Delta$ be a simple closed curve on $\Sigma_{g,1}$ that is isotopic to the boundary component $\partial \Sigma_{g,1}$. We obtain $\Sigma_g$ from $\Sigma_{g,1}$ by capping $\Delta$ with a disc $\hat D$. Let $p$ denote a point in the interior of the capped disc $\hat D$. We have a short exact sequence
\begin{align}
    1 \rightarrow \langle t_{\Delta}\rangle \rightarrow \Mod_{g,1} \rightarrow \Mod_g^1 \rightarrow 1
\end{align}
obtained by extending representatives of mapping classes over $\hat D$ as the identity \cite[Sect.4.2.5]{fm}.

According to \cite[Sect.4.4.1]{fm}, there exists a chain $c_1,..,c_{2g}$ of simple closed curves such that $\Delta$ is the boundary of a regular neighbourhood of $c_1 \cup \cdots \cup c_{2g}$ in $\Sigma_{g,1}$. Setting $t_i = t_{c_i}$, $k = 2g$, and $f_{2k} = t_2\cdots t_k$, we see by (\ref{sep_twist_almost_productsquares}) that $$t_{\Delta} = \prod_{j=0}^{2k-1}t^2_{f_{2k}^j(c_1)} \cdot f_{2k}^{2k}.$$ 

Note that $c_2,..,c_k$ is a chain of simple closed curves of odd length $k-1$, let $d_1$ and $d_2$ denote the boundary components of a regular neighbourhood of $c_2 \cup \cdots \cup c_k$ in $\Sigma_{g,1}$. Then by \cite[Prop.4.12]{fm} we have $$(t_2\cdots t_k)^k = t_{d_1} t_{d_2},$$ and hence $$t_{\Delta} = \prod_{j=0}^{2k-1}t^2_{f_{2k}^j(c_1)} \cdot t_{d_1}^2 \cdot t_{d_2}^2$$ as a product of squares.

Let $C_i$ denote the homology class of $c_i$ in $H_1(\Sigma_g;\F_2)$. One checks that $$f_{2k}(C_i) = \begin{cases}
    C_1 + C_2 ,\; i=1 \\
    C_{i+1},\; 2 \leq i \leq k-1 \\
    C_2 + C_3 + \cdots + C_{k} , \; i = k
\end{cases}.$$ Hence we have $f_{2k}^{k+j}(C_1) = f_{2k}^j(C_1)$ for $1 \leq j \leq k-1$.

Combining the above, we get $$\beta_{\sigma}(t_{\Delta}) = \sum_{j=0}^{2k-1} \overline{f_{2k}^j(C_1)} + 2\overline{D_1} = 2\sum_{j=0}^{k-1}\overline{f_{2k}^j(C_1)} + 2\overline{D_1},$$ where $D_1$ is the homology class of $d_1$ and $d_2$ in $H_1(\Sigma_g ; \F_2)$, hence $$\beta_{\sigma}(t_{\Delta}) \otimes 1 = 0 \in \Img{(\beta)} \otimes_{\Z} \Z/2.$$
\end{proof}

To study the closed case, we use the Birman exact sequence, described as follows. Let $U\Sigma_g$ denote the unit tangent bundle of $\Sigma_g$, then for $g \geq 2$, we have an exact sequence $$1 \rightarrow \pi_1(U\Sigma_g) \xrightarrow{\mathrm{Push}} \Mod_{g,1} \rightarrow \Mod_g \rightarrow 1.$$ Let $\Delta$ denote the boundary component of $\Sigma_{g,1}$. The map $\Mod_{g,1} \rightarrow \Mod_g$ is given by gluing a disc to $\Delta$ and extending mapping class representatives over this disc by the identity. The image of $\pi_1(U\Sigma_g)$ under Push is the disc pushing subgroup. This is the subgroup generated by $t_{\Delta}$ and all genus $g-1$ bounding pair maps in $\Mod_{g,1}$; see \cite{birman69}, \cite[Sect.II]{johnsonbp_genus_1_generates} and \cite[Sect.3]{johnsonbp} for more information.

For $\gamma \in \pi_1(U\Sigma_g)$ and $f \in \Mod_{g,1}$, we have that $\mathrm{Push}(df_*\gamma) = f \; \mathrm{Push}(\gamma) \;f^{-1}$ by \cite[Lem.4]{johnsonbp}, hence we obtain a $\Mod_{g,1}$--equivariant exact sequence
$$1 \rightarrow \pi_1(U\Sigma_g) \xrightarrow{\mathrm{Push}} \Mod_{g,1}[2] \rightarrow \Mod_g[2] \rightarrow 1.$$

\begin{figure}
    \centering
    \includegraphics[width=0.5\linewidth]{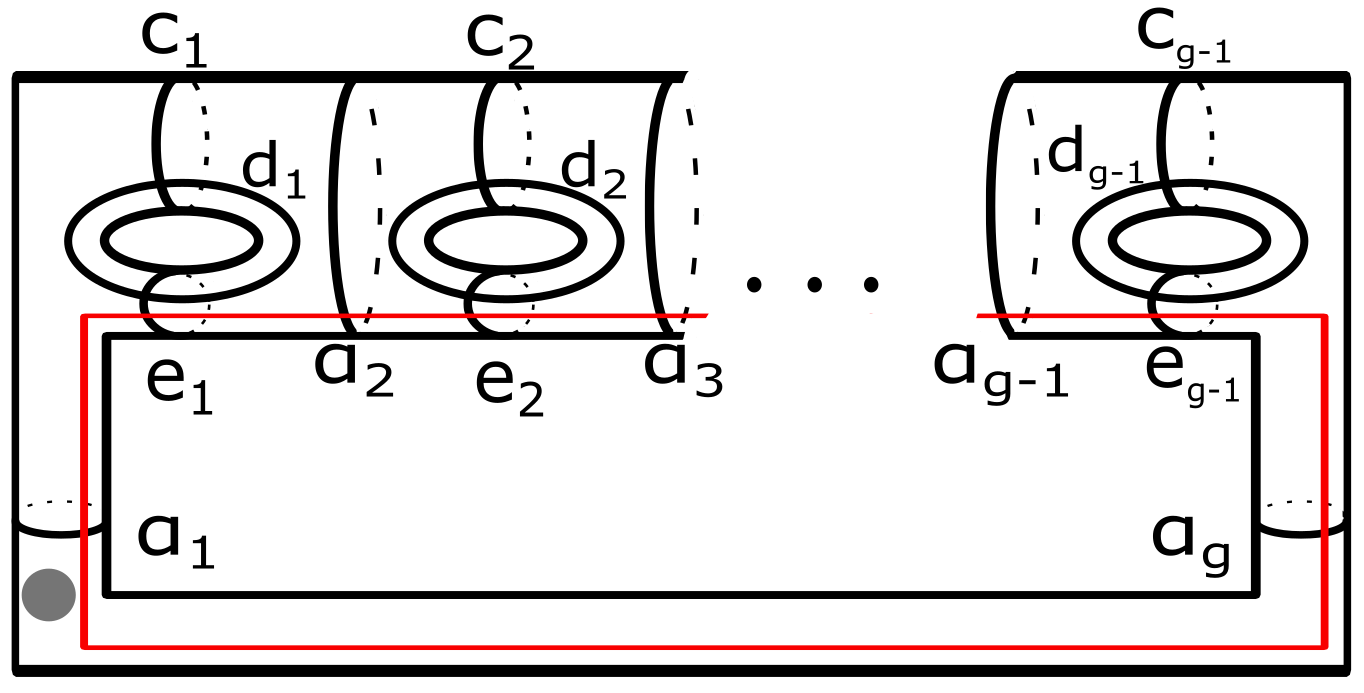}
    \caption{Curves used to factor a genus $g-1$ bounding pair into a product of genus $1$ bounding pairs.}
    \label{genusg_minus1_bp}
\end{figure}

\begin{lemma} \label{closedsurface_level2_sprep}
    The Birman exact sequence induces a short exact sequence $$0 \rightarrow H_1(\Sigma_g;F) \rightarrow H_1(\Mod_{g,1}[2];F) \rightarrow H_1(\Mod_g[2];F) \rightarrow 0$$ of $\mathrm{Sp}_{2g}(\F_2)$--modules.
\end{lemma}
\begin{proof}
    By above, the image of $\mathrm{Push}$ in $\Mod_{g,1}[2]$ is generated by $t_{\Delta}$ and all genus $g-1$ bounding pair maps. Identify $H_1(\Mod_{g,1}[2];\Z)$ with $W_{\sigma,g}$ using Sato's map $\beta_{\sigma}$, and identify $H_1(\Mod_{g,1}[2];F)$ with $W_{\sigma, g} \otimes_{\Z} F$. By the five--term exact sequence associated to the Birman exact sequence, we have a short exact sequence $$H_1(U\Sigma_g;F)_{\Mod_{g,1}[2]} \xrightarrow{\mathrm{Push}_*} H_1(\Mod_{g,1}[2];F) \rightarrow H_1(\Mod_g[2];F) \rightarrow 0.$$ We identify the image of $\mathrm{Push}_*$ with the image of $\mathrm{Push}$ under $\beta_{\sigma}$, as a $\mathrm{Sp}_{2g}(\F_2)$--module. 
    
    We are left to compute the $\mathrm{Sp}_{2g}(\F_2)$--submodule of $W_{\sigma,g} \otimes_{\Z} F$ generated by the images of $t_{\Delta}$ and genus $g-1$ bounding pairs. By the proof of Lemma \ref{oncepunctured_level2_sprep} we have $\beta_{\sigma} (t_{\Delta}) \otimes 1 = 0$. Now we compute the image of genus $g-1$ bounding pairs.

    We follow Johnson \cite[Sect.3]{johnsonbp_genus_1_generates}, and write a genus $g-1$ bounding pair as a product of genus $1$ bounding pairs. By the change-of-coordinates principle, all genus $g-1$ bounding pairs are conjugate in $\Mod_{g,1}$ to $t_{\alpha_1} t_{\alpha_g}^{-1}$, where $\alpha_1,\alpha_g$ are the curves in Figure \ref{genusg_minus1_bp}. Write $$t_{\alpha_1} t_{\alpha_g}^{-1} = t_{\alpha_1} t_{\alpha_2}^{-1}t_{\alpha_2}t_{\alpha_3}^{-1} \cdots t_{\alpha_{g-1}} t_{\alpha_g}^{-1} := \prod_{i=1}^{g-1} t_{\alpha_i} t_{\alpha_{i+1}}^{-1}.$$ Let $c_i,d_i,e_i$ be the simple closed curves on $\Sigma_{g,1}$ pictured in Figure \ref{genusg_minus1_bp}. The boundary of a regular neighbourhood of $c_i \cup d_i \cup e_i$ is the union of $\alpha_i$ and $\alpha_{i+1}$; using the chain relation as in Section \ref{bp_sep_twist_eval_subsection}, we can write $t_{\alpha_i} t_{\alpha_{i+1}}^{-1}$ as a product of squares of Dehn twists. By Corollary \ref{img_torelli_mod2}, we have $$\beta_{\sigma}(t_{\alpha_i}t_{\alpha_{i+1}}^{-1})\otimes 1 = (4\overline{C_i}\;\overline{D_i} \; \overline{A_i}) \otimes 1,$$ where $[c_i] = C_i, \; [d_i] = D_i, [\alpha_{i}] = A_i \in H_1(\Sigma_g;\F_2)$. Note that $A_i = A_{i+1}$ for $1 \leq i \leq g-1$, and that the $\{C_i,D_i \}_{i=1}^{g-1}$ can be completed into a symplectic basis $\{ C_i, D_i \}_{i=1}^g$ for the intersection form; for example, set $C_g = A_1$ and $D_g$ to be the homology class of the red curve of Figure \ref{genusg_minus1_bp}.

    We have \begin{align*}
        \beta_{\sigma}(t_{\alpha_1}t_{\alpha_g}^{-1}) \otimes 1 = \sum_{i=1}^{g-1} \beta_{\sigma}(t_{\alpha_i}t_{\alpha_{i+1}}^{-1})\otimes 1 \\
        = \sum_{i=1}^g (4\overline{C_i} \; \overline{D_i} \; \overline{A_1}) \otimes 1.
    \end{align*} Under the isomorphism \begin{align*}
        \bigwedge^3 H_1(\Sigma_g;\F_2) \rightarrow Z_g \\
        u \wedge v \wedge w \mapsto 4 \overline{u} \; \overline{v} \; \overline{w}
    \end{align*} of $\mathrm{Sp}_{2g}(\F_2)$--modules given in Lemma \ref{level2ab_filtration_submod}, we see that the $\mathrm{Sp}_{2g}(\F_2)$--submodule generated by $\beta_{\sigma}(t_{\alpha_1} t_{\alpha_g}^{-1})\otimes 1$ coincides with the $\mathrm{Sp}_{2g}(\F_2)$--submodule of $\bigwedge^3 H_1(\Sigma_g;\F_2)$ generated by $\omega \wedge A_1$, where $\omega = \sum_{i=1}^g C_i \wedge D_i$ as in (\ref{omega_def}). This is the $\mathrm{Sp}_{2g}(\F_2)$--module $\Img{\epsilon}$, where $\epsilon$ is given by (\ref{epsilon_definition}). By Lemma \ref{wedge3_composition_part2}, we have that the $\mathrm{Sp}_{2g}(\F_2)$ submodule of $W_{\sigma,g} \otimes_{\Z} F$ generated by $\beta_{\sigma}(t_{\alpha_1}t_{\alpha_g}^{-1}) \otimes 1$ is isomorphic to $H_1(\Sigma_g;F)$.
\end{proof}

\section{Action on homology of the level $2$ congruence subgroup of the symplectic group} \label{level2_spgrp_section}

Let $R = \Z$ or $\F_2$, and let $\mathrm{Sp}_{2g}(R)$ denote the symplectic group, which we identify with the isometry group of the intersection form $(H_1(\Sigma_g;R), - \cdot -)$. Let $\mathrm{Sp}_{2g}(\Z)[2]$ denote the kernel of the $\pmod{2}$ reduction map $\Sp_{2g}(\Z) \rightarrow \Sp_{2g}(\F_2)$. The conjugation action of $\mathrm{Sp}_{2g}(\Z)$ on $\mathrm{Sp}_{2g}(\Z)[2]$ induces a linear action of $\mathrm{Sp}_{2g}(\Z)$ on $H_1(\mathrm{Sp}_{2g}(\Z)[2];\F_2)$, so we obtain a representation of $\mathrm{Sp}_{2g}(\Z)$, that factors to a representation of $\mathrm{Sp}_{2g}(\Z)/ \mathrm{Sp}_{2g}(\Z)[2] = \mathrm{Sp}_{2g}(\F_2)$. In this section, we compute the composition factors of this representation using the following.

\begin{lemma} \label{shortexact_sp_mod2}
    Let $F$ be a field of characteristic $2$, regarded as a trival $\mathrm{Sp}_{2g}(\F_2)$--module. Let $\rho : \Mod_{g,1} \rightarrow \mathrm{Sp}_{2g}(\Z)$ denote the symplectic representation. Then the exact sequence
    $$1 \rightarrow \mathcal{I}_{g,1} \rightarrow \Mod_{g,1}[2] \xrightarrow{\rho} \mathrm{Sp}_{2g}(\Z)[2] \rightarrow 1$$
    induces an exact sequence $$H_1(\mathcal{I}_{g,1};F)_{\Mod_{g,1}[2]} \xrightarrow{\iota} H_1(\Mod_{g,1}[2];F) \rightarrow H_1(\mathrm{Sp}_{2g}(\Z)[2];F) \rightarrow 0$$ of $\mathrm{Sp}_{2g}(\F_2)$--modules. Hence the composition factors of $\Img{\iota}$ and $H_1(\Mod_{g,1}[2];F)$ gives the composition factors of $H_1(\mathrm{Sp}_{2g}(\Z)[2];F)$.
\end{lemma}
\begin{proof}
The inclusion $\mathcal{I}_{g,1} \hookrightarrow \Mod_{g,1}[2]$ induces a homomorphism $$\iota: H_1(\mathcal{I}_{g,1};\Z)_{\Mod_{g,1}[2]} \rightarrow H_1(\Mod_{g,1}[2];\Z).$$ Now consider the five--term exact sequence obtained from the exact sequence $$1 \rightarrow \mathcal{I}_{g,1} \rightarrow \Mod_{g,1}[2] \xrightarrow{\rho} \mathrm{Sp}_{2g}(\Z)[2] \rightarrow 1,$$
this is the exact sequence of $\mathrm{Sp}_{2g}(\F_2)$--modules given by
$$H_1(\mathcal{I}_{g,1};\Z)_{\Mod_{g,1}[2]} \xrightarrow{\iota} H_1(\Mod_{g,1}[2];\Z) \rightarrow H_1(\mathrm{Sp}_{2g}(\Z)[2];\Z) \rightarrow 0,$$
see \cite[Exercise 6(a), p.47]{brown_grpcoho} for more details.

Right exactness of the tensor product $- \otimes_{\Z} F$, along with the universal coefficient theorem, gives the result.
\end{proof}

By Corollary \ref{img_torelli_mod2}, the image of $\iota$ in Lemma \ref{shortexact_sp_mod2} is the $\mathrm{Sp}_{2g}(\F_2)$--submodule of $W_{\sigma, g} \otimes_{\Z} \F_2$ generated by $(4\bar C_1 \bar C_2 \bar D_1)\otimes 1$. Here, we start with a chain of simple closed curves $c_1,c_2,c_3$, and set $d_1,d_2$ to be the boundary components of a regular neighbourhood of the union of the $c_i$'s. Then let $C_i = [c_i], D_1 = [d_1] \in H_1(\Sigma_{g,1};\F_2)$. We need the following.

\begin{lemma} \label{cubics_equals_image_torelli}
    Let $Z_g$ denote the span of all vectors of the form $(4 \bar u \bar v \bar w) \otimes 1$ in $W_{\sigma, g} \otimes_{\Z} \F_2,\; (u,v,w \in H_1(\Sigma_g;\F_2))$, as in Lemma \ref{level2ab_filtration_submod}. Then for $F =\F_2$, the image of the map $\iota$ in Lemma \ref{shortexact_sp_mod2} is $Z_g$.
\end{lemma}
\begin{proof}
    Let $\{X_i\}_{i=1}^{2g}$ be a symplectic basis for $H_1(\Sigma_g;\F_2)$ with $X_{2i-1} \cdot X_{2j} = \delta_{ij}$. By Corollary \ref{img_torelli_mod2}, we have that $4 \bar X_1 \bar X_2 \bar X_3$ generates $\Img{\iota}$ as a $\mathrm{Sp}_{2g}(\F_2)$--module. Using the notation of Section \ref{second_diff_subsection}, take the factor mix for the blocks $\{ 1, 2\}$ and $\{ 5,6\}$, that is, take the map
    $$T: X_1 \mapsto X_1 + X_6, \; X_2 \mapsto X_2, \; X_5 \mapsto X_5 + X_2,\; X_6 \mapsto X_6$$
    where $T$ fixes all other basis vectors. Then we have $$4\bar X_1 \bar X_2 \bar X_3 + T(4\bar X_1 \bar X_2 \bar X_3) = 4 \bar X_2 \bar X_3 \bar X_6.$$ Using factor rotations and swaps, we obtain all the $4 \bar X_i \bar X_j \bar X_k$ with $\{i,j,k\}$ isotropic. Similarly, the action by factor swaps and rotations on $4 \bar X_1 \bar X_2 \bar X_3$ obtains all the $4 \bar X_i \bar X_j \bar X_k$ where $\{i,j,k\}$ contains one block. Hence $\Img{\iota}$ contains all the basis vectors $4 \bar X_i \bar X_j \bar X_k$ for $X_i,X_j,X_k$ distinct.
\end{proof}

\begin{cor} \label{compfactors_ab_level2_sp}
    Let $F$ be a field of characteristic $2$, regarded as a trivial $\mathrm{Sp}_{2g}(\Z)[2]$--module. Let $g \geq 3$ be odd, then the composition factors of $H_1(\mathrm{Sp}_{2g}(\Z)[2];F)$ as an $\mathrm{Sp}_{2g}(\F_2)$--module are $H_1(\Sigma_g;F), F, \ker{\delta_2}$ with multiplicities $1,1,1$ respectively. Let $g \geq 3$ be even, then the composition factors for $H_1(\mathrm{Sp}_{2g}(\Z)[2];F)$ as an $\mathrm{Sp}_{2g}(\F_2)$--module are $H_1(\Sigma_g; F), F, \ker{\delta_2}/\langle \omega \rangle$ with multiplicities $1,2,1$. Here, $\delta_2$ and $\omega$ are defined in (\ref{contraction_maps_definition}) and (\ref{omega_def}), using the symplectic space $(H_1(\Sigma_g;F), Q)$.
\end{cor}
\begin{proof}
    For the case $F = \F_2$, we have, by Lemmas \ref{cubics_equals_image_torelli} and \ref{level2ab_filtration_submod} that $Z_g = \Img{\iota} = \bigwedge^3 H_1(\Sigma_g;\F_2)$ as a $\mathrm{Sp}_{2g}(\F_2)$--module. So $\Img{\iota}$ has composition factors $H_1(\Sigma_g;\F_2), \ker{\delta_3}/ \Img{\epsilon}$ with multiplicities $2,1$ respectively if $g$ is odd. And $\Img{\iota}$ has composition factors $H_1(\Sigma_g;\F_2), \ker{\delta_3}$ with multiplicities $1,1$ respectively if $g$ is even, by Theorem \ref{kerd3_irreducible}.

    Now, if $g$ is odd, the composition factors of $W_{\sigma, g} \otimes_{\Z} \F_2$ are $H_1(\Sigma_g;\F_2), \F_2, \ker{\delta_2} , \ker{\delta_3}/\Img{\epsilon}$ with multiplicities $3,1,1,1$ respectively. If $g$ is even, the composition factors of $W_{\sigma, g} \otimes_{\Z} \F_2$ are $H_1(\Sigma_g;\F_2), \F_2, \ker{\delta_2}/\langle \omega \rangle , \ker{\delta_3}$ with multiplicities $2,2,1,1$ respectively. Hence Lemma \ref{shortexact_sp_mod2} gives the result.

    For the general case, use the field extension $\F_2 = \{0,1\} \subset F$, exactness of $- \otimes_{\F_2} F$, and the universal coefficient theorem.
\end{proof}

\section{Action on flat line bundles} \label{line_bundles_section}
In this section, we use our description of the $\mathrm{Sp}_{2g}(\F_2)$-module $H_1(\Mod_{g,1}[2];\Z)$ to study the action of $\Mod_{g,b}^p$ on flat line bundles for the orbifold $(\mathcal{T}_{g,b}^p, \Mod_{g,b}^p[2])$, where $\mathcal{T}_{g,b}^p$ denotes the Teichmuller space of $\Sigma_{g,b}^p$. We first recall relevant definitions of line bundles over orbifolds, then go on to study the action of subgroups of $\Mod_{g,b}^p$ by pullback on flat line bundles.

\subsection{Line bundles over orbifolds}

For the definitions of Picard groups given below, see \cite[Sect.2]{putmanduke} for more information. Let $X$ be a simply-connected CW complex, and let $G$ be a discrete group acting properly discontinuously on $X$ by homeomorphisms. We assume that $(X,G)$ is a good orbifold in the sense of \cite[Sect.2.3.1]{putmanduke}. A $G$-equivariant line bundle is a $\C$-line bundle $\rho:E\rightarrow X$ with an action of $G$ on $E$ such that $\rho (g \hat x)  = g\rho(\hat x)$ and the induced maps $\rho^{-1}(x) \rightarrow \rho^{-1}(gx)$ are $\C$-linear. Let $\mathrm{Pic}_{\mathrm{top}}(X,G)$ denote the group of $G$-equivariant line bundles up to equivalence, where $\rho_1:E_1 \rightarrow X$ is equivalent to $\rho_2:E_2 \rightarrow X$ if there exists a $G$-equivariant isomorphism $E_1 \rightarrow E_2$ such that 
\[\begin{tikzcd}
	{E_1} & {E_2} \\
	X
	\arrow[from=1-1, to=1-2]
	\arrow["{\rho_1}"', from=1-1, to=2-1]
	\arrow["{\rho_2}", from=1-2, to=2-1]
\end{tikzcd}\] 
commutes and the maps $\rho_1^{-1}(x) \rightarrow \rho_2^{-1}(x)$ are linear. If $G$ acts on $X$ freely, then $\mathrm{Pic}_{\mathrm{top}}(X,G) = \mathrm{Pic}_{\mathrm{top}}(X/G)$, where the right hand side is the ordinary Picard group of line bundles over $X/G$. In general, there is an isomorphism $c_1:\mathrm{Pic}_{\mathrm{top}}(X,G) \rightarrow H^2_G(X;\Z)$ and the torsion subgroup is given by $H^2_G(X;\Z)^{tor} \cong \mathrm{Hom}_{\Z}(H_1(G;\Z)^{tor}, \C^*)$. Here, $H^*_G(X;\Z)$ denotes the equivariant cohomology of $(X,G)$ as in \cite[Sect.2.1]{putmanduke}.

\textit{Flat line bundles.} Let $\phi:H_1(G;\Z) \rightarrow \C^*$ be a homomorphism. Let $\phi':G \rightarrow \C^*$ be the composition of $\phi$ with the abelianization map $G \rightarrow H_1(G;\Z)$. The flat line bundle on $(X,G)$ defined by $\phi$ is $X \times \C$ with $G$-action $$g(x,z) = (gx, \phi'(g)z).$$ By \cite[Lem.2.10]{putmanduke}, the group $\mathrm{Pic}(X,G)^{tor}$ is the set of elements that can be given a flat line bundle structure.

\textit{Pullbacks.} Suppose $G$ is a normal subgroup of a discrete group $\Gamma$, where $\Gamma$ extends the action of $G$ on $X$, and $\Gamma$ acts properly discontinuously by homeomorphisms and forms a good orbifold $(X,\Gamma)$. Let $E \rightarrow X$ be a $G$-equivariant line bundle, and let $\gamma \in \Gamma$. The pullback $\gamma^*E$ has fiber $\gamma^*E_x = E_{\gamma x}$. To define a $G$-action on $\gamma^*E$, let $h \in G$; we need a linear map $\gamma^*E_x = E_{\gamma x} \rightarrow \gamma^*E_{hx} = E_{\gamma h x}$. We have $\gamma h x = \gamma h \gamma^{-1} \gamma x$, and $\gamma h \gamma^{-1} \in G$. Define $$h(x,e) = (hx, \gamma h \gamma^{-1} e),$$
for $h \in G, x \in X, e\in E_{\gamma x}$.

For a flat line bundle $E$ coming from a character $\chi:G \rightarrow \C^*$ as above, we see that $\gamma^*E$ is the flat line bundle coming from the character $\chi_{\gamma}:G\rightarrow \C^*$, where $\chi_{\gamma}(h) = \chi(\gamma h \gamma^{-1})$. The $\Gamma/G$-invariant flat line bundles are elements of $$\mathrm{Hom}_{\Z}(H_1(G;\Z)^{\mathrm{tor}}, \C^*)^{\Gamma/G} \cong \mathrm{Hom}_{\Z}(H_1(G;\Z)^{\mathrm{tor}}_{\Gamma}, \C^*).$$

\subsection{Line bundles over moduli space}

Let $\Sigma_{g,b}^p$ denote a compact oriented surface of genus $g$ with $b$ boundary components and $p$ marked points in its interior. Let $\mathrm{Diff}^+(\Sigma_{g,b}^p, \partial \Sigma_{g,b}^p)$ denote the topological group consisting of orientation-preserving diffeomorphisms of $\Sigma_{g,b}^p$ that fix the boundary pointwise and permute the marked points. Let $\Mod_{g,b}^p = \pi_0(\mathrm{Diff}^+(\Sigma_{g,b}^p, \partial \Sigma_{g,b}^p))$ denote the mapping class group. Let $\mathcal{T}_{g,b}^p$ denote the Teichmuller space of $\Sigma_{g,b}^p$; this gives a contractible space on which $\Mod_{g,b}^p$ acts properly discontinuously. Define $\Mod_{g,b}^p[2]$ to be the kernel of the action of $\Mod_{g,b}^p$ on $H_1(\Sigma_{g,b};\F_2)$.

Suppose $b+p \leq 1$, we consider the orbifold $(\mathcal{T}_{g,b}^p, \Mod_{g,b}^p[2])$ with the induced action of a subgroup $\mathrm{Stab}_{g,b}^p[q]$ of $\Mod_{g,b}^p$ that fixes a spin structure; the following lemma shows that the action by pullback always fixes a flat line bundle. Recall that spin structures on $\Sigma_g$ are in natural bijection with the set of \textit{Symplectic quadratic forms} $q:H_1(\Sigma_g;\F_2) \rightarrow \F_2$ that satisfy $q(x+y) = q(x) + q(y) + x \cdot y$. 

\begin{lemma} \label{spin_mcg_fixes_flat_line_bundle} Suppose that $b       +p \leq 1$ and $g \geq 3$.
    Let $q:H_1(\Sigma_g;\F_2) \rightarrow \F_2$ be a Symplectic quadratic form. Let $\mathrm{Stab}_{g,b}^p[q]$ denote the subgroup of $\mathrm{Mod}_{g,b}^p$ consisting of elements $f \in \mathrm{Mod}_{g,b}^p$ that satisfy $f^*q = q$. There exists a flat line bundle of order $2$ in $\mathrm{Pic}(\mathcal{T}_{g,b}^p, \mathrm{Mod}_{g,b}^p[2])$ that is left invariant under the action of $\mathrm{Stab}_{g,b}^p[q]$ by pullback.
\end{lemma}
\begin{proof}
    We show the case $(b,p) = (1,0)$ first.
    Let $\{X_i\}_{i=1}^{2g}$ be a symplectic basis for $H_1(\Sigma_g;\F_2)$. Let $$\mathcal{B} = \{ \overline{X_i} = \overline{X_i} \otimes 1,\; \overline{X_{ij}} = (2\overline{X_i} \overline{X_j}) \otimes 1 \; (i<j),\; \overline{X_{ijk}} = (4 \overline{X_i} \overline{X_j} \overline{X_k}) \otimes 1\; (i<j<k)\}$$ denote the basis for $W_{\sigma,g} \otimes_{\Z} \Z/2 = H_1(\mathrm{Mod}_{g,1}[2];\F_2)$. Let $v = \sum_i X_i \in H_1(\Sigma_g;\F_2)$, and let $S_v = \{ i \; | \; v_i=1\}$ be the support of $v$, then by Lemma \ref{general_matrix_formula} $$\overline{v} = \sum_{i \in S_v} \overline{X_i} + \sum_{i<j \in S_v} \overline{X_{ij}} + \sum_{i<j<k \in S_v} \overline{X_{ijk}}.$$ Let $Q = (q_{ij})_{i,j=1}^{2g}$ denote a matrix with $q(v) = v^TQv = \sum_{i,j} v_i q_{ij}v_j = \sum_{i<j \in S_v}(q_{ij}+q_{ji}) + \sum_{i \in S_v} q_{ii}$. Using the basis $\mathcal{B}$, define an $\F_2$-linear map $q: H_1(\mathrm{Mod}_{g,1}[2];\F_2) \rightarrow \F_2$ that sends $\overline{X_i}$ to $q_{ii}$, $\overline{X_{ij}}$ to $q_{ij} + q_{ji}$ for $i<j$, and sends $\overline{X_{ijk}}$ to $0$. Then $$q(\overline{v}) = q(v).$$ We have $q(f \overline{v}) = q(\overline{fv}) = q(fv) = q(v)$ for all $f$ in the stabilizer of the form $q$. Hence we obtain a flat line bundle from the character $$\chi_q: \mathrm{Mod}_{g,1}[2] \twoheadrightarrow H_1(\mathrm{Mod}_{g,1}[2];\F_2) \xrightarrow{q} \F_2$$ such that for all $f \in \mathrm{Stab}_{g,1}[q]$ we have $\chi(ft_c^2f^{-1}) = \chi(t_{f(c)}^2) = q(\overline{f_*[c]}) = q(f_*[c]) = q([c]) = \chi( t_c^2)$. Since $\mathrm{Mod}_{g,1}[2]$ is generated by squares of Dehn twists, we are done. 
    
    The case $(b,p) = (0,1)$ also follows from the argument above since Lemma \ref{oncepunctured_level2_sprep} shows that the map $\Mod_{g,1}^0[2] \rightarrow \Mod_{g,0}^1[2]$ obtained by capping the boundary component of $\Sigma_{g,1}^0$ with a marked disk induces an isomorphism on $H_1(-;\F_2)$. For the case $(b,p) = (0,0)$, Lemma \ref{closedsurface_level2_sprep} identifies $H_1(\Mod_g[2];\F_2)$ with a quotient of $H_1(\Mod_{g,1}[2];\F_2)$ as a $\mathrm{Sp}_{2g}(\F_2)$-module. The proof of Lemma \ref{closedsurface_level2_sprep} shows that the image of $\mathrm{Push}:\pi_1(U\Sigma_g) \rightarrow H_1(\Mod_{g,1}[2];\F_2)$ is contained in the span of the $\overline{X_{ijk}}$ for $i<j<k$, hence the homomorphism $q:H_1(\Mod_{g,1}[2];\F_2) \rightarrow \F_2$ above descends to $H_1(\Mod_g[2];\F_2)$.
\end{proof}

\begin{theorem} \label{spin_mcg_flat_line_uniqueness} Suppose that          $(b,p) = (1,0)$ and $g \geq 3$, or suppose that $b+p \leq 1$ and $g     \geq 9$.
    Let $q:H_1(\Sigma_g;\F_2)\rightarrow \F_2$ be a Symplectic quadratic form. We have that $$H_1(\Mod_{g,b}^p[2];\Z)_{\mathrm{Stab}_{g,b}^p[q]} = \Z/2,$$ hence the flat line bundle constructed in Lemma \ref{spin_mcg_fixes_flat_line_bundle} is the unique nontrivial torsion element in $\mathrm{Pic}(\mathcal{T}_{g,b}^p,\Mod_{g,b}^p[2])$ fixed by the action of $\mathrm{Stab}_{g,b}^p[q]$.
\end{theorem}
\begin{proof}
    Let $O(q)$ denote the group of elements $f \in \Sp_{2g}(\F_2)$ with $f^*q = q$. The short exact sequence $$1 \rightarrow \Mod_{g,b}^p[2] \rightarrow \mathrm{Stab}_{g,b}^p[q] \rightarrow O(q) \rightarrow 1$$
    induces the exact sequence
    $$H_2(O(q);\Z) \rightarrow H_1(\Mod_{g,b}^p[2];\Z)_{\mathrm{Stab}_{g,b}^p[q]} \rightarrow H_1(\mathrm{Stab}_{g,b}^p[q];\Z) \rightarrow H_1(O(q);\Z) \rightarrow 0.$$
    By \cite[Example 1.14]{Randal_Williams_spin_picard}, \cite[Thm.2.14]{Randal_Williams_spinmcg_stability}, and \cite[Thm.A and 8.4]{Sierra_stability_spinmcg}, we have $H_1(\mathrm{Stab}_{g,b}^p[q];\Z) = \Z/4$. We will prove the claim by showing that $H_1(O(q);\Z) = \Z/2$ and $H_2(O(q);\Z) = 0$. 
    
    Let $\Omega(q)$ denote the commutator subgroup of $O(q)$. The Dickson invariant gives a homomorphism $D: O(q) \rightarrow \Z/2$ that sends $f$ to $\mathrm{rank}(I-f) \pmod{2}$; see \cite[Ch.11]{taylor_classical_groups}. By \cite[Thm.11.47]{taylor_classical_groups} we have $\Omega(q) = [\Omega(q),\Omega(q)]$, and by \cite[Thm.11.51]{taylor_classical_groups} $\mathrm{ker}(D) = \Omega(q)$. The five term exact sequence of the exact sequence $1 \rightarrow \Omega(q) \rightarrow O(q) \rightarrow \Z/2 \rightarrow 1$ then implies that $H_1(O(q);\Z) \cong \Z/2$. To calculate $H_2(O(q);\Z)$, use the Lyndon--Hochschild--Serre spectral sequence for $\Omega(q) \hookrightarrow O(q) \twoheadrightarrow \Z/2$. We have $E^2_{p, 2-p} = H_p(\Z/2;H_{2-p}(\Omega(q);\Z))$, giving $E^2_{0,2} = H_2(\Omega(q);\Z)$, $E^2_{1,1} = 0$, and $E^2_{2,0} = 0$. The group $\Omega(q)$ has Schur multiplier $0$, hence $E^2_{0,2} = 0$; see \cite[Ch.3.8]{Wilson_finite_simple_groups}. We conclude that $H_2(O(q);\Z) = 0$.
\end{proof}

\section{Congruence subgroups of the automorphism group of a free group} \label{cong_autfn_section}
Let $p$ be a prime, let $\F_p$ denote the field with $p$ elements, and let $\overline{\F}_p$ denote its algebraic closure. In this section, we investigate the $\mathrm{GL}_n(\F_p)$--module structure on the groups $H_1(\Aut(F_n)[p];\F_p)$ and $H_1(\Aut(F_n)[p];\overline{\F}_p)$, and prove periodic representation stability results. The composition factors we find are labeled using dominant weights for the algebraic group $\mathrm{SL}_n \overline{\F}_p$. We begin by reviewing the groups we consider, along with the root datum associated to $\mathrm{SL}_n \overline{\F}_p$.

Let $F_n$ be a free group of rank $n$, let $\Aut(F_n)$ denote the automorphism group of $F_n$, and denote the abelianization of $F_n$ by $H$. The abelianization morphism induces a surjective representation $$\rho: \mathrm{Aut}(F_n) \rightarrow GL_n(\Z).$$ Let $IA_n$ be the kernel of $\rho$, then we get a short exact sequence $$1 \rightarrow IA_n \rightarrow \Aut(F_n) \xrightarrow{\rho} GL_n(\Z) \rightarrow 1.$$ The conjugation action of $\Aut (F_n)$ on $IA_n$ induces an action of $\Aut (F_n)$ on $H_1(IA_n;\Z)$ that factors to an action of $\Aut(F_n)/IA_n = GL_n(\Z)$ on $H_1(IA_n;\Z)$. 

Let $\tau': IA_n \rightarrow \Hom_{\Z}(H,\wedge^2H) = H^* \otimes_{\Z} \wedge^2 H$ denote the \textit{Johnson homomorphism}, defined as follows: for $x \in F_n$, let $[x] \in H$ denote its image in the abelianization. For $f \in IA_n$ and $x \in F_n$, we have $f(x)x^{-1} \in [F_n,F_n]$. There is a natural surjection \begin{align*}
    q: [F_n,F_n] \rightarrow \wedge^2 H \\
    [a,b] \mapsto [a] \wedge [b]
\end{align*} with kernel $[F_n, [F_n,F_n]]$. Let $f \in IA_n$, then the assignment \begin{align*}
    F_n \rightarrow \wedge^2 H \\
    x \mapsto q(f(x)x^{-1})
\end{align*}is a homomorphism that induces a homomorphism $\tau'_f:H \rightarrow \wedge^2 H$. Define 
\begin{align*}
    \tau': IA_n \rightarrow \Hom_{\Z}(H, \wedge^2 H) \\
    f \mapsto \tau'_f.
\end{align*}By work of Cohen--Pakianathan \cite{cohen_pakianathan}, Farb \cite{farb_torelli_aut}, and Kawazumi \cite{kawazumi_magnus}, $\tau'$ induces a $GL_n\Z$--equivariant isomorphism $$\tau: H_1(IA_n;\Z) \rightarrow H^* \otimes_{\Z} \wedge^2H,$$ where $H^* = \mathrm{Hom}_{\Z}(H,\Z)$.

Let $\pi_p:GL_n(\Z) \rightarrow \mathrm{GL}_n(\F_p)$ be the natural homomorphism induced by reduction mod $p$. Let $\Gamma(n,p)$ denote the kernel of $\pi_p$. The isomorphism $\tau'$ above induces a homomorphism $$\tau_p': IA_n \rightarrow (H^* \otimes_{\Z} \wedge^2 H) \otimes_{\Z} \Z/p.$$ By \cite[Section 3]{satoh_congruence}, $\tau_p'$ induces a $\mathrm{GL}_n(\F_p)$--equivariant isomorphism $$\tau_p: H_1(IA_n;\Z)_{\Gamma(n,p)} \rightarrow (H^* \otimes_{\Z} \wedge^2 H) \otimes_{\Z} \Z/p,$$ where $H_1(IA_n;\Z)_{\Gamma(n,p)}$ is the module of coinvariants for the action of $\Gamma(n,p)$ on $H_1(IA_n;\Z)$, which is naturally a $GL_n(\Z) / \Gamma(n,p) = \mathrm{GL}_n(\F_p)$--module. 

Write $V = H_1(F_n;\F_p) = H \otimes_{\Z} \F_p$, and identify $(H^* \otimes_{\Z} \wedge^2 H) \otimes_{\Z} \F_p = V^* \otimes \wedge^2 V$ as $\mathrm{GL}_n(\F_p)$--modules.

Let $\Aut(F_n)[p]$ denote the kernel of $\pi_p \circ \rho$, then we obtain an exact sequence $$1 \rightarrow IA_n \rightarrow \Aut(F_n)[p] \xrightarrow{\rho} \Gamma(n;p) \rightarrow 1.$$ The five--term exact sequence induces a short exact sequence 
\begin{equation} \label{five_term_cong_aut}
0 \rightarrow H_1(IA_n;\Z)_{\Gamma(n,p)} \xrightarrow{\eta} H_1(\Aut(F_n)[p];\Z) \rightarrow H_1(\Gamma(n,p);\Z) \rightarrow 0
\end{equation}
of $\mathrm{GL}_n(\F_p)$--modules; see \cite[Ch.II, Sect.5, Ex. 6(a)]{brown_grpcoho} for more details; the morphism $\eta$ is injective by \cite[Section 3]{satoh_congruence}. Here, the action of $\mathrm{GL}_n(\F_p)$ is induced from the action of $\Aut(F_n)$ by conjugation, using $\pi_p \circ \rho$.  The exact sequence (\ref{five_term_cong_aut}) gives the following.

\begin{cor}\label{aut_GL_extension}
    Let $p$ be a prime. We have two exact sequences of $\mathrm{GL}_n(\F_p)$--modules $$0 \rightarrow H_1(IA_n;\Z)_{\Gamma(n,p)}  \rightarrow H_1(\Aut(F_n)[p];\F_p) \rightarrow H_1(\Gamma(n,p);\F_p) \rightarrow 0,$$
    and $$0 \rightarrow H_1(IA_n;\Z)_{\Gamma(n,p)} \otimes_{\F_p} \overline{\F}_p  \rightarrow H_1(\Aut(F_n)[p];\overline{\F}_p) \rightarrow H_1(\Gamma(n,p);\overline{\F}_p) \rightarrow 0.$$ Here, $H_1(IA_n;\Z)_{\Gamma(n,p)}$ is isomorphic to $V^* \otimes \wedge^2V$ as a $\mathrm{GL}_n(\F_p)$--module, where $V = \F_p^n$ is the tautological representation. Hence, the composition factors for the $\mathrm{GL}_n(\F_p)$--modules $$H_1(IA_n;\Z)_{\Gamma(n,p)}, H_1(\Gamma(n,p);\F_p)$$ give the composition factors for $H_1(\Aut(F_n)[p];\F_p)$.
\end{cor}
\begin{proof}
    We have that $H_1(\Gamma(n,p);\Z)$ consists entirely of $p$--torsion by Lee--Szczarba \cite[Thm.1.1]{lee_szczarba}. We have that $H_1(\Aut(F_n)[p];\Z)$ consists entirely of $p$--torsion by Satoh \cite[Thm.1.1]{satoh_congruence}. Hence all groups appearing in the short exact sequence (\ref{five_term_cong_aut}) consist entirely of $p$--torsion. Applying the functor $-\otimes_{\Z} \F_p$ then gives the exact sequence $$0 \rightarrow H_1(IA_n;\Z)_{\Gamma(n,p)} \otimes_{\Z} \F_p \xrightarrow{\eta \otimes \Id} H_1(\Aut(F_n)[p];\Z) \otimes_{\Z} \F_p \rightarrow H_1(\Gamma(n,p);\Z) \otimes_{\Z} \F_p \rightarrow 0.$$
    The first exact sequence now follows from the universal coefficient theorem. To obtain the second exact sequence from the first, use the fact that $-\otimes_{\F_p} \overline{\F}_p$ is exact, along with the universal coefficient theorem.
\end{proof}

\subsection{Labelling simple $\mathrm{SL}_n$--modules} \label{irreps_label_type_A} 

Here we fix a labelling of simple $\overline{\F}_p[\mathrm{SL}_n]$--modules. Fix a maximal torus $T$ of $\mathrm{SL}_n\overline{\F}_p$ to be the subgroup of diagonal matrices, and let $X(T)$ denote the character group of $T$. Consider the identification $X(T) = \Z[L_1,..,L_n]/ \langle L_1+ \cdots + L_n \rangle$ where $L_i:T \rightarrow \overline{\F}_p^*$ sends $\mathrm{diag}(c_1,..,c_n)$ to $c_i$, and write multiplication of characters using additive notation. Fix the simple system $\Delta = \{ \alpha_i = L_i - L_{i+1} \; | \; 1 \leq i \leq n-1 \}$ for the root system $\Phi \subset X(T)$. Denote the dominant weights with respect to $\Delta$ by $X(T)^+$, and the fundamental dominant weights corresponding to $\alpha_i$ by $\omega_i$. For a dominant weight $\lambda \in X(T)^+$, denote the rational irreducible $\mathrm{SL}_n\overline{\F}_p$--module with highest weight $\lambda$ by $L(\lambda)$.

By \cite[Prop.4.2.2]{mcninch} we have $L(\omega_i) = \wedge^i V$ for $1 \leq i \leq n-1$, where $V = \overline{\F}_p^n$ is the tautological representation of $\mathrm{SL}_n\overline{\F}_p$. Furthermore, $V^*$ is isomorphic to $\wedge^{n-1}V$ as a $\mathrm{SL}_n\overline{\F}_p$--module, hence $V^* = L(\omega_{n-1})$. We also have that $L(\omega_i) = V(\omega_i)$, where $V(\omega_i)$ is the Weyl module associated to $\omega_i$; see \cite[Sect.2.1]{mcninch} for the definition of a Weyl module. Note that $L(0)$ is the $1$--dimensional trivial representation of $\mathrm{SL}_n$. We use the following lemma of McNinch later on.

\begin{lemma} \cite[Prop.4.6.10]{mcninch} \label{compo_factor_SL}
    Let $V = \overline{\F}_p ^n$, then the following holds as $SL(V)$--modules:
    \begin{itemize}
        \item If $n \equiv 1 \pmod{p}$ then $V^* \otimes \wedge^2V$ has composition factors $L(\omega_1), L(\omega_2+\omega_{n-1}), L(\omega_1)$.
        \item If $n \neq 1 \pmod{p}$ then $V^* \otimes \wedge^2V$ has composition factors $L(\omega_2+\omega_{n-1}), L(\omega_1)$.
        \item If $n \equiv 0 \pmod{p}$ then $V^* \otimes V$ has composition factors $L(0), L(\omega_1 + \omega_{n-1}), L(0)$.
        \item If $n \neq 0 \pmod{p}$ then $V^* \otimes V$ has composition factors $L(\omega_1 + \omega_{n-1}), L(0)$.
    \end{itemize}
\end{lemma}

All the modules $L(\lambda)$ for $\lambda \in X(T)^+$ remain simple when restricted to $\mathrm{SL}_n(\F_p)$; see \cite[Thm.2.11]{humphreys_modular_reps}. Hence we can label simple $\overline{\F}_p[\mathrm{SL}_n(\F_p)]$--modules as above. In our statements, we consider simple $\F_p[\mathrm{SL}_n(\F_p)]$--modules which remain simple as $\overline{\F}_p[\mathrm{SL}_n(\F_p)]$--modules after extension of scalars, hence we can label them using dominant weights as above.

To find composition factors, we use the following observation: let $M$ be a $\F_p[\mathrm{SL}_n(\F_p)]$--module, and suppose $M$ has a filtration by submodules $$0 = M_0 \subset M_1 \subset \cdots \subset M_k = M$$ with $M_{i+1}/M_i$ simple for all $i$. For a field extension $\F_p \hookrightarrow \overline{\F}_p$, we have that $\overline{\F}_p$ is a flat as an $\F_p$--module. For the exact sequence of $\F_p$ modules $$0 \rightarrow M_i \rightarrow M_{i+1} \rightarrow M_{i+1}/M_i \rightarrow 0$$ flatness of $\overline{\F}_p$ gives an exact sequence $$0 \rightarrow M_i \otimes_{\F_p} \overline{\F}_p \rightarrow M_{i+1}\otimes_{\F_p} \overline{\F}_p \rightarrow M_{i+1}/M_i\otimes_{\F_p} \overline{\F}_p \rightarrow 0$$
so we have $$M_{i+1}\otimes_{\F_p} \overline{\F}_p / M_i \otimes_{\F_p} \overline{\F}_p \cong (M_{i+1}/M_i) \otimes_{\F_p} \overline{\F}_p.$$ Each $M_i \otimes_{\F_p} \overline{\F}_p$ becomes a $\overline{\F}_p[\mathrm{SL}_n(\F_p)]$--module via $g \cdot m \otimes c = (g \cdot m) \otimes c$. We use the following observation.

\begin{prop} \label{filtration_extending_scalars}
    Let $M$ be a $\F_p[\mathrm{SL}_n(\F_p)]$--module, and suppose $M$ has a filtration by submodules $$0 = M_0 \subset M_1 \subset \cdots \subset M_k = M.$$ Suppose the $\overline{\F}_p[\mathrm{SL}_n(\F_p)]$--modules $(M_{i+1}/M_i)\otimes_{\F_p} \overline{\F}_p$ are simple for all $i$, then the $\F_p[\mathrm{SL}_n(\F_p)]$--modules $M_{i+1}/M_i$ are simple for all $i$.
\end{prop}
\begin{proof}
    Suppose there exists a proper $\F_p[\mathrm{SL}_n(\F_p)]$--submodule $W \subset M_{i+1}/M_i$, then as $- \otimes_{\F_p} \overline{\F}_p$ is exact, $W \otimes_{\F_p} \overline{\F}_p \subset M_{i+1}/M_i \otimes_{\F_p} \overline{\F}_p$ remains proper and invariant under the action of $\mathrm{SL}_n(\F_p)$, contradicting the assumption.
\end{proof}

\subsection{Composition factors}
Following Corollary \ref{aut_GL_extension}, we compute the composition factors of the $\mathrm{SL}_n$--modules $H_1(IA_n;\Z)_{\Gamma(n,p)}, H_1(IA_n;\Z)_{\Gamma(n,p)} \otimes_{\F_p} \overline{\F}_p, H_1(\Gamma(n,p);\F_p), H_1(\Gamma(n,p);\overline{\F}_p)$.

\textbf{Composition factors of $H_1(\Gamma(n,p);\F_p)$ and  $H_1(\Gamma(n,p);\overline{\F}_p)$.} We use the following result of Lee--Szczarba. Let $M(n,p)$ denote the abelian group of $n \times n$ matrices with entries in $\F_p$, and let $M_0(n,p)$ be the subgroup of matrices with trace zero. Define a map \begin{align*}
    \varphi: \Gamma(n,p) \rightarrow M_0(n,p) \\
    A \mapsto \frac{1}{p}(A-I) \pmod{p}.
\end{align*}
Then $\varphi$ induces an isomorphism of $\mathrm{SL}_n(\F_p)$--modules $$H_1(\Gamma(n,p);\Z) \rightarrow M_0(n,p),$$ where $\mathrm{SL}_n(\F_p)$ acts on $M_0(n,p)$ by conjugation \cite[Thm.1.1]{lee_szczarba}.

\begin{lemma} \label{compo_factor_SL_traceless_matrices}
    Let $\overline{M_0}(n,p)$ denote the abelian group of traceless $n \times n$ matrices over $\overline{\F}_p$. Let $\mathrm{SL}_n(\F_p)$ act on both $\overline{M_0}(n,p)$ and $M_0(n,p)$ by conjugation. Then the composition factors for both $\mathrm{SL}_n(\F_p)$ modules are given as follows.
    \begin{itemize}
        \item If $p$ divides $n$, the composition factors are $L(0), L(\omega_1 + \omega_{n-1})$.
        \item If $p$ does not divide $n$, then the composition factors are $L(\omega_1 + \omega_{n-1})$.
    \end{itemize}
\end{lemma}
\begin{proof}
Let $F$ be a field, and let $V$ be a vector space over $F$. Then $\Hom_F(V,V)$ becomes an $SL(V)$--module via $g \cdot \alpha = g \circ \alpha \circ g^{-1}$, for $g \in SL(V), \alpha \in \Hom_F(V,V)$. Furthermore, $Hom_F(V,V)$ is isomorphic to $V^* \otimes V$ as a $SL(V)$--module by the map 
\begin{align*}
    \Psi:V^* \otimes V \rightarrow \Hom_F(V,V) \\
    f \otimes v \mapsto [w \mapsto f(w)v].
\end{align*}
Let $\{ e_1,..,e_n\}$ be a basis for $V$, and let $\{e_1^*,..,e_n^*\}$ denote the dual basis for $V^*$. Then $\Psi(\sum_{i=1}^n e_i^* \otimes e_i) = \Id_V \in \Hom_F(V,V)$. Since $\Hom_F(V,V)$ has a unique $1$--dimensional $SL(V)$--submodule spanned by $\Id_V$, we get that $\sum_{i=1}^n e_i^* \otimes e_i$ spans the unique $1$--dimensional $SL(V)$--submodule of $V^* \otimes V$. Note that the basis $\{ e_i\}$ identifies $Hom_F(V,V)$ with the abelian group of $n \times n$ matrices over $F$. 

Consider the map \begin{align*}
    \xi: V^* \otimes V \rightarrow F \\
    f \otimes v \mapsto f(v).
\end{align*}
Then $\xi$ is a surjective morphism of $SL(V)$--modules, where $SL(V)$ acts trivially on $F$. With the isomorphism $\Psi$ above, the map $\xi$ corresponds to the trace map. Hence, after picking a basis, we can identify the abelian group of traceless $n \times n$ matrices over $F$ with the kernel of $\xi$. When $F$ has characteristic $p>0$, we have $$\xi(\sum_{i=1}^n e_i^* \otimes e_i) = n \equiv 0 \pmod{p}$$ if and only if $p$ divides $n$. Hence, if $p$ divides $n$, we have the following filtration by $SL(V)$--modules $$0 \subset \langle \sum_{i=1}^n e_i^* \otimes e_i \rangle \subset \ker \xi \subset V^* \otimes V,$$ and if $p$ does not divide $n$, we have the filtration $$0 \subset \ker \xi \subset V^* \otimes V.$$

If $F = \overline{\F}_p$, then Lemma \ref{compo_factor_SL} implies that the filtrations above have maximal length, hence form a composition series. If $p$ divides $n$, we obtain $\langle \sum_{i=1}^n e_i^* \otimes e_i \rangle = L(0)$ and $\ker \xi / \langle \sum_{i=1}^n e_i^* \otimes e_i \rangle = L(\omega_1 + \omega_{n-1})$. If $p$ does not divide $n$, we obtain $\ker \xi = L(\omega_1 + \omega_{n-1})$. The restriction theorem \cite[Thm.2.11]{humphreys_modular_reps} implies that these filtrations continue being composition series when restricted to $\mathrm{SL}_n(\F_p)$.

Since the constructions above were independent of the field, we obtain an analogous result for the $\mathrm{SL}_n(\F_p)$ modules $\ker \xi$ over $F = \F_p$, by Proposition \ref{filtration_extending_scalars}, the filtrations above give a composition series for $\mathrm{SL}_n(\F_p)$ over $\F_p$.
\end{proof}

\begin{cor} \label{comp_factors_SL_congruence_autf}
    For the $\mathrm{SL}_n(\F_p)$--modules $H_1(\Gamma(n,p);\F_p), H_1(\Gamma(n,p);\overline{\F}_p)$, we have the following.
    \begin{itemize}
        \item If $p$ divides $n$, the composition factors are $L(0), L(\omega_1 + \omega_{n-1})$.
        \item If $p$ does not divide $n$, the composition factors are $L(\omega_1 + \omega_{n-1})$.
    \end{itemize}
\end{cor}
\begin{proof}
    Follows from Lemma \ref{compo_factor_SL_traceless_matrices} and \cite[Thm.1.1]{lee_szczarba}.
\end{proof}

\textbf{Composition factors of $H_1(IA_n;\Z)_{\Gamma(n,p)}, H_1(IA_n;\Z)_{\Gamma(n,p)} \otimes_{\F_p} \overline{\F}_p$.} Recall that we have identified $H_1(IA_n;\Z)_{\Gamma(n,p)}$ and $V^* \otimes \wedge^2 V$ as $\mathrm{SL}_n(\F_p)$--modules, where $V = \F_p^n$. The action of $\mathrm{SL}_n(\F_p)$ on $H_1(IA_n;\Z)_{\Gamma(n,p)} \otimes_{\F_p} \overline{\F}_p$ is obtained from that on $V^* \otimes \wedge^2 V$ by extension of scalars.
\begin{cor} \label{coinvariants_autfn_composition_factors}
    For the $\mathrm{SL}_n(\F_p)$--modules $H_1(IA_n;\Z)_{\Gamma(n,p)}, H_1(IA_n;\Z)_{\Gamma(n,p)} \otimes_{\F_p} \overline{\F}_p$, we have the following.
    \begin{itemize}
        \item If $n = 1 \pmod{p}$, then the composition factors are $L(\omega_1), L(\omega_2+\omega_{n-1}), L(\omega_1)$.
        \item If $n \neq 1 \pmod{p}$, then the composition factors are $L(\omega_2 + \omega_{n-1}), L(\omega_1)$.
    \end{itemize}
\end{cor}
\begin{proof}
    It follows from Lemma \ref{compo_factor_SL} that the $\mathrm{SL}_n \overline{\F}_p$--module $(V^* \otimes \wedge^2 V) \otimes_{\F_p} \overline{\F}_p$ has the composition series mentioned in the statement. By the restriction theorem \cite[Thm.2.11]{humphreys_modular_reps}, the restriction to $\mathrm{SL}_n(\F_p)$ has the same composition factors.

    Let $F$ be a field of characteristic $p>0$, and let $V$ be a vector space of dimension $n$ over $F$. The map 
    \begin{align*}
        \kappa: V^* \otimes \wedge^2 V \rightarrow V \\
        f \otimes v \wedge w \mapsto f(v)w-f(w)v
    \end{align*}
    is a surjective morphism of $GL(V)$--modules; see \cite[Exercise B.15]{fulton_harris}. Let $\{e_i\}_{i=1}^n$ be a basis for $V$, and let $\{e_i^*\}_{i=1}^n$ be the dual basis for $V^*$. The map 
    \begin{align*}
        \tau: V \rightarrow V^* \otimes \wedge^2 V \\
        e_k \mapsto \sum_{i=1}^n e_i^* \otimes e_k \wedge e_i
    \end{align*}
    is $GL(V)$--equivariant. To see this, let $A = (a_{ij}) \in GL(V)$, where $A \cdot e_j = \sum_{i=1}^n a_{ij} e_i$, then $A \cdot e_j^* = \sum_{l=1}^n a^{jl}e_l^*$, where $A^{-1} = (a^{ij})$. Then \begin{align*}
        A \cdot \sum_j e_j^* \otimes e_k \wedge e_j = \sum_j(\sum_l a^{jl}e_l^*) \otimes A \cdot e_k \wedge (\sum_{i} a_{ij} e_i) \\
         = \sum_{(i,l)} \sum_j a_{ij}a^{jl} e_l^* \otimes A \cdot e_k \wedge e_i \\
         = \sum_{(i,l)} (AA^{-1})_{il} e_l^* \otimes A\cdot e_k \wedge e_i \\
         = \sum_{(i,l)} \delta_{il} e_l^* \otimes A \cdot e_k \wedge e_i \\
         = \sum_i e_i^* \otimes A \cdot e_k \wedge e_i.
    \end{align*}
    Furthermore, we have that $$\kappa \circ \tau (e_k) = -(n-1)e_k,$$ so $\tau (e_k) \in \ker \kappa$ if and only if $n-1 \equiv 0 \pmod{p}$.

    In summary, when $n-1 \equiv 0 \pmod{p}$, we have the following filtration of $GL(V)$--modules $$0 \subset \Img \tau \subset \ker \kappa \subset V^* \otimes \wedge^2 V,$$ and when $n-1 \neq 0 \pmod{p}$ we have the filtration $$0 \subset \ker \kappa \subset V^* \otimes \wedge^2 V.$$
    Specializing $F = \overline{\F}_p$, and using Lemma \ref{compo_factor_SL}, we see that in both cases, we have a Jordan--Holder series for the $SL(V)$--module $V^* \otimes \wedge^2 V$. Specializing $F = \F_p$ and using Proposition \ref{filtration_extending_scalars} implies the filtration also gives a Jordan--Holder series for this case.
\end{proof}

\begin{cor} \label{congruence_autfn_composition_factors}
    Let $p$ be a prime. For the $\mathrm{SL}_n(\F_p)$--modules $H_1(\Aut(F_n)[p];\F_p)$ and $H_1(\Aut(F_n)[p];\overline{\F}_p)$, we have the following.
    \begin{itemize}
        \item If $n = 1 \pmod{p}$, then the composition factors are $L(\omega_1), L(\omega_2 + \omega_{n-1}), L(\omega_1), L(\omega_1 + \omega_{n-1})$.
        \item If $n = 0 \pmod{p}$, then the composition factors are $L(\omega_2+\omega_{n-1}), L(\omega_1), L(0), L(\omega_1 + \omega_{n-1})$.
        \item If $n \neq 0,1 \pmod{p}$, then the composition factors are $L(\omega_2+\omega_{n-1}), L(\omega_1) , L(\omega_1 + \omega_{n-1})$.
    \end{itemize}
    Hence, these modules give a stably periodic sequence in the sense of Church--Farb \cite[Def.8.1]{church_farb}.
\end{cor}
\begin{proof}
    Follows from Corollaries \ref{aut_GL_extension}, \ref{comp_factors_SL_congruence_autf}, and \ref{coinvariants_autfn_composition_factors}.
\end{proof}

\printbibliography
\end{document}